%% file: main.tex
\documentclass[10pt]{article}
\usepackage[
  a4paper, mag=1000,
 left=3.5cm, right=3.5cm, top=2.5cm, bottom=2.5cm, headsep=0.7cm, footskip=1cm
]{geometry}
\usepackage[T2A]{fontenc}
\usepackage[utf8]{inputenc}
\usepackage[english]{babel}
\ifpdf\usepackage{epstopdf}\fi
\usepackage{hyphenat}
\usepackage{setspace}
\usepackage{enumerate}
\usepackage{amssymb, amsmath, hyperref}
\usepackage{amsthm}
\usepackage{verbatim, tikz, wrapfig}
\usetikzlibrary{patterns}
\usepackage[linecolor = black, backgroundcolor = white]{todonotes}

\newcounter{theorems}
\newtheorem{lem}[theorems]{Lemma}
\newtheorem{prop}[theorems]{Proposition}
\newtheorem{cor}[theorems]{Corollary}
\newtheorem{theorem}[theorems]{Theorem}

\newtheorem{defin}{Definition}
\newtheorem{rem}[theorems]{Remark}
\newtheorem{remark}[theorems]{Remark}

\def \CO {\mathfrak C}
\def \CI {\mathfrak C^1}
\def \TT {\mathbb T^2}
\def \NN {\mathbb N}
\def \RR {\mathbb R}
\def \eps {\varepsilon}
\def \UK {UK}
\def \FI {{F_{init}}}
\def \NI {{N_{init}}}
\def \di {\delta_{init}}
\def \FB {F_{Bow}}
\def \DD {\mathbb T^2 \setminus UK}
\def \Homeo {\mathrm{Homeo}}
\def \HS {\mathcal H}
\def \Lcal {\mathcal L}
\DeclareMathOperator{\supp}{supp}
\DeclareMathOperator{\Leb}{Leb}
\def \dh {\partial_h} 
\def \dv {\partial_v} 
\def \Dop {Z(H)} 
\DeclareMathOperator{\intt}{int}
\def \FPL {F_\infty}
\def \FPLL {F_L}
\def \FL {F_{Lin}}
\def \LL {\mathbb T^2 \setminus R}
\def \KK {K}
\DeclareMathOperator{\Diff}{Diff}
\DeclareMathOperator{\dist}{dist}
\def \FBL {\Lcal_\Pi(F_0)}
\def \fpl {f_L}
\def \phy {\varphi}
\def \FH {F_{\HS}}

\def \CT {C_{thick}}
\def \K {\mathcal K}
\DeclareMathOperator{\leb}{Leb}

\makeatletter
\def\blfootnote{\gdef\@thefnmark{}\@footnotetext}
\makeatother

\begin{document}
\title{A $C^1$ Anosov diffeomorphism with a horseshoe that attracts almost any point
\footnote{Supported in part by the RFBR, 16-01-00748-a.}}
\author{C. Bonatti\footnote{Universit\'e de Bourgogne, Dijon}, S. S. Minkov\footnote{Brook Institute of Electronic Control Machines, Moscow \newline{} \indent\;\, Moscow Center for Continuous Mathematical Education, Moscow}, A. V. Okunev\footnote{National Research University Higher School of Economics, Moscow}, I. S. Shilin\footnote{National Research University Higher School of Economics, Moscow}}
\date{}

\maketitle
\begin{abstract}
We present an example of a $C^1$ Anosov diffeomorphism of a two-torus with a physical measure such that its basin has full Lebesgue measure and its support is a horseshoe of zero measure.
\end{abstract}



\section{Introduction}
Consider a diffeomorphism $F$ that preserves a probability measure $\nu$.  The \emph{basin} of $\nu$ is the set of all points $x$ such that the sequence of measures
\[
\delta^n_x := \frac 1 n (\delta_x + \dots + \delta_{F^{n-1}(x)}) 
\]
converges to $\nu$ in the weak-$*$ topology. 
A measure is called \emph{physical} if its basin has positive Lebesgue measure. It is well known (see~\cite[$\S1.3$]{BDV}) that any transitive $C^2$ Anosov diffeomorphism (note that all known Anosov diffeomorphisms are transitive) has a unique physical measure and
\begin{itemize}
\item the basin of this measure has full Lebesgue measure,
\item the support of this measure coincides with the whole phase space.
\end{itemize}

However, in $C^1$ dynamics there are many phenomena that are impossible in $C^2$. For instance, Bowen~\cite{Bow} has constructed an example of a $C^1$ diffeomorphism of the plane with a \emph{thick} horseshoe (i.e., a horseshoe with positive Lebesgue measure) and Robinson and Young~\cite{RY} have embedded a thick horseshoe in a $C^1$ Anosov diffeomorphism of $\TT$. Using an analog of their constructions, we prove the following theorem.
\begin{theorem} \label{t:thm}
There exists a $C^1$ Anosov diffeomorphism $\FH$ of the 2-torus $\TT$ that admits a physical measure $\nu$ such that
\begin{itemize}
\item its basin has full Lebesgue measure,
\item it is supported on a horseshoe $\HS$ of zero Lebesgue measure, 
\item $\omega(x) = \HS$ for Lebesgue-a.e. $x\in\TT$. 
\end{itemize}
\end{theorem}

\noindent Note that $\FH$ is transitive, since by~\cite{New} any Anosov diffeomorphism of $\TT$ is transitive. 

In our construction the horseshoe $\HS$ will be \emph{semithick}, i.e., $\HS$ will be a product of a Cantor set with positive measure in the unstable direction and a Cantor set with zero measure in the stable direction. The dynamics on $\HS$ in the unstable direction will be as in Bowen's thick horseshoe, and in the stable direction as in a linear Smale's horseshoe. Let us also remark that the last statement of the theorem is not a corollary of the former two. Indeed,  $\delta^n_x \to \nu$ does not imply that $\omega(x) \subset \supp \nu$, since the orbit of $x$ can spend almost all the time near $\supp \nu$ but still be far from this set infinitely many times.

The last statement of Theorem~\ref{t:thm} claims that, despite $\FH$ is transitive, the horseshoe is the "attractor" of $\FH$ if we are interested in the behavior of typical (with respect to the Lebesgue measure) points. This intuition corresponds to the definition of the Milnor attractor (\cite{Mil}), which we now recall. 
\begin{defin}
Consider a diffeomorphism $F: N \to N$ of a compact riemannian manifold $N$; the metric induces the \emph{Lebesgue measure} on $N$. The \emph{Milnor attractor} of $F$ (notation: $A_M(F)$) is the minimal by inclusion closed subset of $N$ that contains $\omega(x)$ for a.e. $x \in N$ with respect to the Lebesgue measure. 
\end{defin}
\noindent The Milnor attractor always exists, as proved in~\cite{Mil}. In our case, $A_M(\FH) = \HS$, by Theorem~\ref{t:thm}. Let us mention that another interesting example of a transitive map with nontrivial Milnor attractor is the intermingled basins example (\cite[\S 11.1.2]{BDV}). 

As mentioned above, any transitive $C^2$ Anosov diffeomorphism has a unique physical measure with full support, so its Milnor attractor is the whole manifold. This also holds for $C^1$ generic transitive Anosov diffeomorphisms, as the following remark claims. We will say that some property is possessed by a \emph{generic} diffeomorphism in some class if diffeomorphisms with this property form a residual subset in this class.

\begin{remark} \label{r:gen-Anosov}
For a generic transitive $C^1$-smooth Anosov diffeomorphism, the set of points whose positive semi-orbit is everywhere dense has full Lebesgue measure.
\end{remark}
\begin{proof}
For a diffeomorphism $F\colon M \to M$ and an open set $U \subset M$, denote by $A(U) = U \cup F^{-1}(U) \cup \dots$ the set of points whose positive semi-orbit intersects~$U$. Let us show that for any $U$ and $\eps > 0$ the set $\mathcal A(U, \eps)$ of diffeomorphisms such that the Lebesgue measure of $A(U)$ is greater than $1-\eps$ is an open and dense subset of the (open) set of transitive Anosov diffeomorphisms. 

Recall that any $C^2$-smooth transitive Anosov diffeomorphism has a unique physical measure that attracts almost every point (w.r.t. the Lebesgue measure). So for such maps the positive semi-orbit of almost any point is dense. Since $C^2$-diffeomorphisms are dense in $\Diff^1(M)$, the set $\mathcal A(U, \eps)$ is dense.

Let us show that $\mathcal A(U, \eps)$ is open. Let $F \in \mathcal A(U, \eps)$. Then for some number $N$ the measure of the set $U \cup F^{-1}(U) \cup \dots \cup F^{-N}(U)$ is greater than $1-\eps$. Since the set $U$ is open, the same is true for any diffeomorphism sufficiently close to~$F$, which proves the openness.

Choose a countable base of topology $(U_n)_{n\in\mathbb N}$ on~$M$. The set $\bigcap_{n, m} \mathcal A(U_n, \frac{1}{m})$ is residual, and for every $F \in \bigcap_{n, m} \mathcal A(U_n, \frac{1}{m})$ the positive semi-orbit of almost every point is dense. 
\end{proof}

Since the definition of the Milnor attractor uses the Lebesgue measure, the Milnor attractor may be changed after conjugation by a homeomorphism. This often represents a big technical difficulty. Consider, for example, partially hyperbolic skew products over Anosov diffeomorphisms. A small perturbation of such skew product is conjugated to another skew product by a homeomorphism (see \cite{IN} and references therein). However, this conjugacy can change the Milnor attractor. 

Let us say that the Milnor attractor of a $C^1$ diffeomorphism $F$ is \emph{topologically invariant in $C^1$}, if for any $C^1$-diffeomorphism $G$ such that $G$ and $F$ are conjugated by a homeomorphism ($G=H \circ F \circ H^{-1}$) we have $A_M(G) = H(A_M(F))$. The first example of a non topologically invariant Milnor attractor was constructed by S. Minkov in his thesis and is similar to the example from~\cite[$\S 3$]{KRM}. Our construction gives an open set of such examples: 
\begin{cor}
For any $C^1$-diffeomorphism that is sufficiently close to $\FH$, the Milnor attractor is not topologically invariant in $C^1$. 
\end{cor}
\begin{proof}
Since $\FH$ is structurally stable, it has a $C^1$-neighborhood $U$ such that all maps in $U$ are conjugate to each other. Since $C^2$-diffeomorphisms are dense in the space of $C^1$-diffeomorphisms, one may find a $C^2$ Anosov diffeomorphism $G \in U$. Any map $F \in U$ is conjugate to both $\FH$ and $G$. As discussed above, any Anosov diffeomorphism of $\TT$ is transitive, and for any transitive $C^2$ Anosov diffeomorphism the Milnor attractor is the whole manifold. So $A_M(G) = \TT$. Since $A_M(\FH) = \HS$, the attractor of $F$ is not topologically invariant in $C^1$.
\end{proof}

\section{Plan of the paper}
The rest of the text is devoted to the proof of Theorem~\ref{t:thm}.
To prove Theorem~\ref{t:thm}, we  first define a certain class $\CI$ that consists of $C^1$-smooth Anosov diffeomorphisms with a semi-thick horseshoe and then prove that a generic diffeomorphism in $\CI$ satisfies the conclusion of Theorem~\ref{t:thm}. To construct the class $\CI$, we will first describe a $C^1$-smooth Anosov diffeomorphism $\FI: \TT \to \TT$ that has an invariant semi-thick horseshoe~$\HS$. This description is done in section~\ref{s:construction}; we also state several properties of $\FI$ and define certain subsets of the torus ($\UK$, $S$, $R, \dots$) that will be needed to define the class~$\CI$. The proof of the existence of a map $\FI$ with required properties will be given in section~\ref{s:Finit}.

Denote by $S$ the local stable foliation of the semi-thick horseshoe $\HS$ for the diffeomorphism~$\FI$. The set $S$ is a product of a thick Cantor set and a segment, and therefore it has positive Lebesgue measure.
One might expect that Lebesgue almost all points will have their images inside $S$ after a sufficient number of forward iterations and so will be attracted to~$\HS$. Unfortunately, this is not necessarily true: one could arrange a thick horseshoe far from $\HS$ --- then every its point would not be attracted to~$\HS$. 
However, it is still reasonable to expect that almost every point gets inside $S$ for \emph{a generic} diffeomorphism with a semi-thick horseshoe~$\HS$. This motivates our genericity argument presented in section~\ref{s:Baire}. We think that presenting a constructive example would have led to far more technical difficulties.

The class $\CI$ is introduced in section~\ref{s:Baire}.
For every map in the class $\CI$ the dynamics on the local stable foliation $S$ of the semi-thick horseshoe is the same as for $\FI$. 
For $F \in \CI$, denote by $B(F) = \cup_{n\ge 0} F^{-n}(S)$ the basin of attraction  of the semi-thick horseshoe~$\HS$. In section~\ref{s:Baire} we will show that any diffeomorphism $F \in \CI$ such that $B(F)$ has full Lebesgue measure has the three properties from Theorem~\ref{t:thm}. 
To prove that such diffeomorphisms are generic in~$\CI$, we consider, for every $\eps > 0$, the set $A^1_\eps \subset \CI$ of all $F \in \CI$ such that $\leb(B(F)) > 1-\eps$, where $\leb$ is the probability Lebesgue measure on~$\TT$. We will show that these sets are open and dense in~$\CI$. Thus, the set $A^1_0 = \cap_{n=1}^\infty A^1_{1/n}$ is residual. But for any $F \in A^1_0$ one has $\leb(B(F)) = 1$, so every such $F$ has the properties stated in Theorem~\ref{t:thm}. 

In section~\ref{s:Baire} we also introduce an auxiliary class $\CO \supset \CI$ of \emph{homeomorphisms} of $\TT$ and the analogues $A_\eps$ of the sets $A^1_\eps$ for this class. This class is used to prove that $A^1_\eps$ is dense in $\CI$. It is convenient to begin section~\ref{s:Baire} with the definition of $\CO$, then define $\CI$ using $\CO$, and then prove the main theorem modulo the density and openness of the sets~$A^1_\eps$. 

In sections~\ref{s:open} and~\ref{s:dense} we prove that the set $A^1_\eps$ is open and dense, respectively. To obtain the density we use two technical lemmas which are proved later in section~\ref{s:technical}.

\input construction.tex
\input baire.tex
\input openness.tex

\input density.tex
\input F_initExists.tex
\input calculus.tex

\end{document}

%% file: construction.tex
\section{Diffeomorphism \texorpdfstring{$\FI$}{F-init} with a semi-thick horseshoe} \label{s:construction}
In this section we will describe a $C^1$-smooth Anosov diffeomorphism $\FI: \TT \to \TT$ with a semi-thick horseshoe~$\HS$ and define the following subsets of $\TT$ which will play a crucial role in our further construction:
\begin{description}
\item [$\UK$] is a small rectangle around the horseshoe~$\HS$,
\item [$S\subset \UK$] is the union of all local stable fibers of the points of~$\HS$, 
\item [$R$] is a neighborhood of $\UK$ such that outside $R$ the diffeomorphism $\FI$ coincides with a linear Anosov diffeomorphism~$\FL$.
\end{description}
The complete description of $\FI$ and these subsets is a little cumbersome, because it contains many technical properties. 

\subsection{Constructing \UK}
Consider a large number $\NI\in\mathbb{N}$.
Let $\FL$ be a linear Anosov diffeomorphism of $\TT = \mathbb R^2/\mathbb Z^2$ defined by the matrix  $\bigl(\begin{smallmatrix}
2&1 \\ 1&1
\end{smallmatrix} \bigr)^\NI$. Fix the metric on $\TT = \RR^2/\mathbb Z^2$ that comes from $\RR^2$.
Let us call \emph{vertical} the unstable direction of~$\FL$; respectively, the stable direction will be called \emph{horizontal}. All maps of the torus which we consider in this paper preserve the unstable foliation of the diffeomorphism $\FL$. We will call this foliation \emph{vertical} as well. 
We will use the following elementary property of linear Anosov diffeomorphisms of~$\TT$:
\begin{remark} \label{r:dense}
For arbitrary $\eps>0$, if $\NI$ is large enough, the fixed points of $\FL$ form an $\eps$-net on $\TT$.
\end{remark}

\begin{prop} \label{p:horseshoe}
There exists a rectangle $\KK$ with vertical and horizontal edges such that the following holds.
\begin{itemize}
\item
The restriction of $\FL$ to $\KK$ is a linear Smale horseshoe map that takes two horizontal stripes $H_0$ and $H_1$ (adjacent respectively to the lower and the upper horizontal edges of the rectangle) to two vertical stripes $V_0$ and $V_1$. 
\item In the opposite corners of $\KK$ there are two fixed points $p_0$ and $p_1$. 
\item $\FL(\KK) \cap \KK = V_0 \cup V_1$. 
\item
The diameter of $\KK$ can be made arbitrarily small by choosing a large~$\NI$.
\end{itemize}
\end{prop}
\begin{proof}
Choose two sufficiently close (make $\NI$ larger and use Remark~\ref{r:dense} if necessary) fixed points $p_0$ and $p_1$ of the map $\FL$ so that there are no other fixed points in the rectangle $K$ with vertical and horizontal edges and with opposite vertices at $p_0$ and $p_1$. Then $\FL(K) \cap K$ contains two vertical stripes $V_0$ and $V_1$ adjoint to the vertical edges of~$K$. Set $H_0 = \FL^{-1}(V_0)$ and $H_1 = \FL^{-1}(V_1)$. It suffices to show that $\FL(K) \cap K = V_0 \cup V_1$. Suppose there is another connected component $V$ of $\FL(K) \cap K$. Then the horizontal stripe $\FL^{-1}(V)$ is linearly mapped onto $V$, so there must be a fixed point inside~$V$. This contradiction finishes the proof.
\end{proof}

\begin{figure} 
\begin{center}
\begin{tikzpicture}[scale=1.15]
\draw [fill=gray!20, draw=black] (0,0) rectangle (8,1);
\draw [fill=gray!20, draw=black] (0,3) rectangle (8,4);

\draw[pattern=north west lines, pattern color=gray] (1.5 ,0) rectangle (3.3 ,4);
\draw[pattern=north west lines, pattern color=gray] (6.5 ,0) rectangle (4.7 ,4);
\draw (0, 1) -- (0, 3);
\draw (8, 1) -- (8, 3);
%\draw (2, 0) -- (2, 4);
%\draw (6, 0) -- (6, 4);
\node [left] at (0, 0.5) {$UH_0$};
\node [left] at (0, 3.5) {$UH_1$};
%\node [above] at (4, 4) {$K$};
\node [below] at (8, 0) {$\UK$};
\node [below] at (2.4, 0) {$UV_0$};
\node [above] at (5.6, 4) {$UV_1$};
%\draw [draw = black, ultra thick] (2, 0) rectangle (6, 4);
%\draw [draw = black, ultra thick] (0, 0) rectangle (8, 4);
\draw [fill] (2, 4) circle[radius=2pt];
\node [below] at (6, 0) {$p_1$};
\node [above] at (2, 4) {$p_0$};
\draw [fill] (6, 0) circle[radius=2pt];
\vspace{-15mm}
\end{tikzpicture}
\end{center}
\caption{The set $\UK$. The vertical stripes of $\UK \setminus K$ are depicted wider than they are.}
\label{f:UH}
\end{figure}
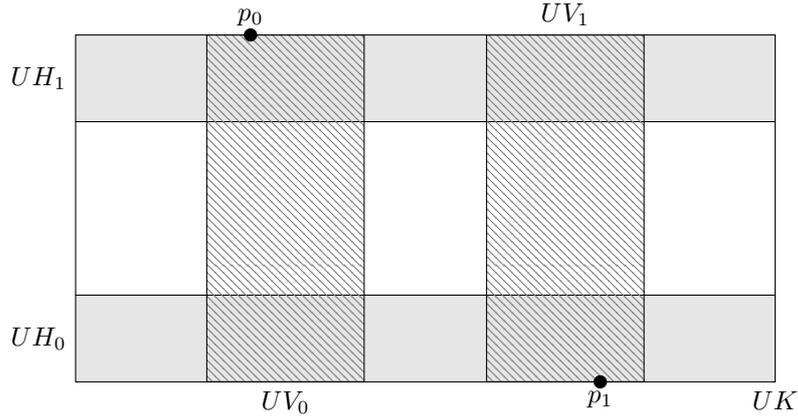
Since $\NI$ is large, we can assume that the edges of the rectangle $\KK$ are all shorter than~$0.01$. Choose a fixed point $q$ of $\FL$ that lies at distance at least $0.3$ from $\KK$. Let us widen the rectangle $\KK$ in the horizontal direction with ratio at most $1.01$ so that both vertical boundary segments of the new rectangle be on the unstable manifold of the point~$q$. Denote the new rectangle by~$\UK$. 

\subsection{Description of the map~\texorpdfstring{$\FI$}{F-init}}
Let $UH_i$ be the continuations of the horizontal stripes $H_i$ into $\UK \setminus K$, and let $UV_i = \FL(UH_i)$ (see fig.~\ref{f:UH}).
Let us construct the map $\FB: UH_0 \cup UH_1 \to UV_0 \cup UV_1$ whose maximal invariant set will be our semi-thick horseshoe~$\HS$. The details will be postponed until section~\ref{s:Finit}, and here we give but a brief description.

The diffeomorphism $\FL$ maps $UH_0$ to $UV_0$ and expands linearly in the vertical direction while contracting linearly in the horizontal one. The new map $\FB$ will map  $UH_0$ onto $UV_0$ too and will also be a direct product of a vertical expansion and a horizontal contraction. The horizontal contraction will be the same as for $\FL$, but the linear vertical expansion will be replaced by a nonlinear one that comes from the Bowen's thick horseshoe construction. 
We define the restriction of $\FB$ to $UH_1$ analogously. 

The horseshoe $\HS = \bigcap \limits_{n = -\infty}^{\infty} \FB^n(UH_0 \cup UH_1)$ is a product of a horizontal thin Cantor set $C_{thin}$ and a vertical thick Cantor set $C_{thick}$. The horseshoe $\HS$ has two fixed points $p_0$ and $p_1$, the same as for the linear horseshoe. Denote by $S$ the union of all local stable leaves inside $\UK$ of the points of $\HS$. The set $S$ is a product of a horizontal segment and the thick Cantor set $C_{thick}$.

Now we introduce some notation required to describe $\FI$.  
In the neighborhood of every point $z \in \TT$ one can consider rectangular coordinates $(x,y)$ where the vertical fibers are parallel to $Oy$-axis and the horizontal ones are parallel to $Ox$. In each tangent plane the lines $dx=\pm dy$ circumscribe a horizontal cone that we will denote~$C_H(z)$. 

Denote by $\tilde R$ the rectangle obtained by widening $\UK$ vertically and horizontally by $0.01$ (center is preserved).

\begin{prop} \label{p:Finit}
For $\NI$ large enough, there exists a $C^1$-smooth Anosov diffeomorphism $\FI: \TT \to \TT$ such that:
\begin{enumerate}
\item
$\FI=\FB$ on $UH_0 \cup UH_1$. Thus, $\FI$ has a semi-thick horseshoe $\HS$.
\item
$\FI=\FL$ outside $\tilde R$. 
\item $\FI(UH_i)=UV_i$. 
\item \label{iv:only-V} $\FI(\UK) \cap \UK = UV_0 \cup UV_1$.
\item the unstable fibers of $\FI$ are strictly vertical.
\item \label{iv:expand} $\FI$ expands the vertical fibers by a factor of at least $1.2$.
%\item У $\FI$ есть неподвижная точка $q$ вне $\UK$, $\partial_v \UK$ лежит в вертикальном слое $q$, и существует прямоугольник $R$ такой, что $\UK$ внутри $R$, а $q$ - снаружи.
\item \label{iv:cones} For any point $x \in \TT$ the image of the interior of the cone $C_H(x)$ under $d\FI$ covers the closure of the cone $C_H(\FI(x))$.
\end{enumerate}
\end{prop}
The proof of this proposition is given in section~\ref{s:Finit} below.

\subsection{Construction of the ``frame''~\texorpdfstring{$R$}{R}}
We will often consider curvilinear quadrilaterals whose boundary is formed by two vertical segments and two curves transverse to the vertical foliation. We will call them ``rectangles''. For such a rectangle $P$ we will call the union of its two vertical boundary segments \emph{the vertical boundary} (notation: $\dv P$) and the union of the other two boundary curves will be called \emph{the horizontal boundary} (notation: $\dh P$).

For the map $\FI$, we will cover the rectangle $\tilde R$ by a curvilinear rectangle $R$ (see fig.~\ref{f:R}) such that 
\begin{itemize}
\item $\dh R$ is contained in the stable fiber of the point $p_0$,
\item $q \not\in R$.
\end{itemize}
Let us construct such rectangle~$R$. The vertical boundary of $R$ will consist of two vertical segments obtained by extending the vertical boundary of $\tilde R$ up and down. The horizontal boundary of $R$ will be formed by two properly chosen pieces of the stable manifold of the point $p_0$ in-between the two segments of the vertical boundary. In order to choose the former, we will need two facts: 1) the stable fiber of $p_0$ is everywhere dense (because every Anosov diffeomorphism of $\TT$ is transitive~\cite{New}), and 2) stable directions of $\FI$ lie inside $C_H$ (because the preimage of the cone field $C_H$ is contained in~$C_H$). 

Let us choose the upper boundary $s$ of the frame $R$ --- the lower one is constructed analogously. Recall that by construction both sides of the rectangle $\tilde R$ are shorter than~$0.05$. Let us move up by $0.03$ from the mid-point of the upper edge of $\tilde R$ and denote the resulting point by $w$. 

Take a point $v$ on $W^s(q_0)$ such that $v$ is $0.001$-close to~$w$. Take as $s$ the  curvilinear segment of $W^s(p_0)$ that contains $v$ and lies in-between the vertical boundary segments of~$R$. All tangents to $s$ lie inside the cones of $C_H$. Therefore, the curve $s$ lies inside a cone with center at $v$, horizontal axis and aperture~$\pi/2$. Thus, over each point of the upper boundary of $\tilde R$ there is a point of the curve~$s$ (and $s$ does not intersect the upper boundary of $\tilde R$). Therefore, $R \supset \tilde R$. On the other hand, the point of $s$ over a point of the vertical boundary of $\tilde R$ will be $0.06$-close to it vertical-wise. Since $q$ was chosen to be $0.3$-far from $K$, this implies $q \not\in R$.

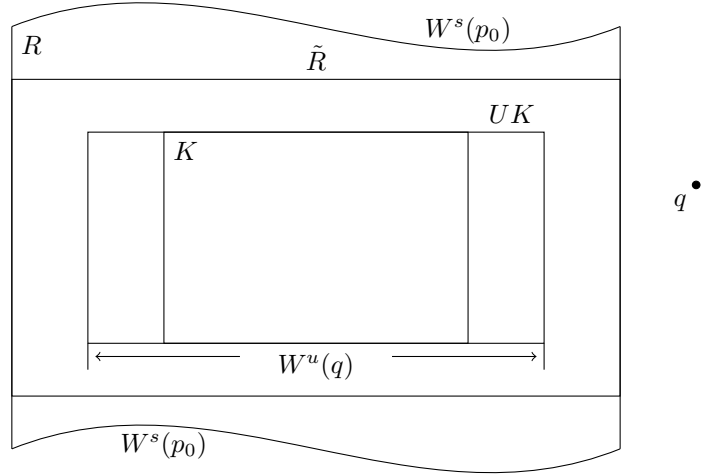
\begin{figure} 
\vspace{-15mm}
\begin{center}
\begin{tikzpicture}[xscale=1, yscale=0.7]
\def \trs{1}
\def \rs{1}
\def \s{0.5}

\draw (0,0) rectangle (4,4);
\draw (-1,0) rectangle (5,4);
\draw (-1-\trs,-\trs) rectangle (5+\trs,4+\trs);
\draw (-1-\trs, -\trs-\rs) -- (-1-\trs, 4+\trs+\rs);
\draw (5+\trs, -\trs-\rs) -- (5+\trs, 4+\trs+\rs);

\draw (-1-\trs, -\trs-\rs) to[out=30, in = 180+30] (5+\trs, -\trs-\rs);
\draw (-1-\trs, 4+\trs+\rs) to[out=30, in = 180+30] (5+\trs, 4+\trs+\rs);

\node[below right] at (0, 4) {$K$};
\node[above left] at (5, 4) {$UK$};
\draw (5, 0) -- (5, -\s);
\draw (-1, 0) -- (-1, -\s);
\node[below] at (2, 0) {$W^u(q)$};
\draw[->] (1, -\s/2) -- (-0.9, -\s/2);
\draw[->] (3, -\s/2) -- (4.9, -\s/2);

\node[above] at (2, 4+\trs) {$\tilde R$};
\node[below right] at (-\trs-\rs, 4+\trs+\rs) {$R$};
\node[above] at (4, 4+\trs+0.5*\rs) {$W^s(p_0)$};
\node[below] at (0, -\trs-0.5*\rs) {$W^s(p_0)$};

\node [draw, circle, fill, inner sep=0pt, minimum size=0.1cm] at (5+2*\trs, 3) {};
\node[below left] at (5+2*\trs, 3) {$q$};

\end{tikzpicture}
\vspace{-15mm}
\end{center}
\caption{The ``frame'' $R$.}
\label{f:R}
\end{figure}

%% file: baire.tex
\section{The Baire genericity argument} \label{s:Baire}
Instead of constructing an explicit example of a diffeomorphism with required properties, we will consider a special class of diffeomorphisms of the torus and prove that a generic map in this class can serve as our example. 
Let us assume that $\FI$, $\FL$, $\UK$, $UH_i$, $UV_i$, $S$, $p_0$, $p_1$, $q$, $R$ are fixed as in section~\ref{s:construction}. Consider any sufficiently large number $L \in \mathbb N$ (we will specify below how large it should be). {\bf The auxiliary class} $\CO$ {\bf of homeomorphisms} consists of homeomorphisms $H : \TT \to \TT$ such that 
\begin{enumerate}
\item \label{i:S} $H|_S=\FI|_S$;
\item \label{i:skew} $H$ preserves the vertical foliation;
\item \label{i:bilip} $H$ is $L$-bi-Lipschitz;
\item \label{i:lin} $H = F_{Lin}$ outside $R$;
\item \label{i:expand} The restriction of $H$ to any vertical fiber expands (nonuniformly) with a factor at least $1.1$ and at most $L$;
\item \label{i:to-D} $H(\intt \UK) \cap \intt \UK = \intt UV_0 \cup \intt UV_1$;
\item \label{i:dh-1} $H^{-L} = \FI^{-L}$ on $\dh \UK$.
\end{enumerate}
\noindent Recall that for the map $\FI$ the boundary $\dv \UK$ is contained in the unstable fiber of the point~$q$ and $\dh R$ is in the stable fiber of the point~$p_0$. Moreover, $\dh \UK$ contains the point $p_0$ together with some segment of its stable fiber. Let us choose the number $L$ to be so large that
\begin{enumerate}[i]
\item \label{i:v-vodu-1} The set $\dv \UK$ is contained in a ball in the vertical fiber of the point $q$ with radius $L$ and center at $q$; 
\item \label{i:v-vodu-2} The ball of radius $L \times 1.1^{-L}$ centered at $q$ lies outside $R$;
\item\label{i:dh-2} $\dh R \subset \dh \FI^{-L}(\UK)$;
\item $\FI$ is $L/2$-bi-Lipschitz.
\end{enumerate}
Let us introduce several properties of maps from the class~$\CO$ that follow from the ones above.
\begin{enumerate}
\setcounter{enumi}{7}
\item \label{i:v-vodu} For $N \ge L$ one has $H^{-N}(\dv \UK) \subset \LL$.
\end{enumerate}

\begin{proof} Since $H = \FL$ outside $R$, the point $q$ is a fixed point of~$H$. It follows from properties~\ref{i:v-vodu-1} and~\ref{i:expand} that $H^{-N} \dv UK$ is contained in the ball in the vertical fiber of $q$ with radius $L \times 1.1^{-N}$ and center at $q$. By property~\ref{i:v-vodu-2} this ball is contained in~$\LL$.
\end{proof}

\begin{enumerate}
\setcounter{enumi}{8}
\item \label{i:dh} $\dh R \subset \dh H^{-L}(\UK)$.
\end{enumerate}
\begin{proof}
This follows straightforwardly from properties~\ref{i:dh-1} and~\ref{i:dh-2}.
\end{proof}

Let us now define {\bf the special class of diffeomorphisms} $\CO^1$.
In order to do that, first choose a small $\di \in (0, 0.1)$ such that 
\begin{enumerate}
\item 
the inequality $\dist_{C^1}(G, \FI) \le \di$ implies that $G$ is conjugate to $\FI$;

\item
for any point $x$ for any linear operator $A: \RR^2 \to \RR^2$ that is $\di$-close to $d_x \FI$ one has:
\begin{equation} \label{e:di}
||A|| < L, \; \|A^{-1}\| < L, \; A(C_H) \supset C_H.
\end{equation}
\end{enumerate}

It is obvious that $\FI \in \CO$. Now we can define the special class of diffeomorphisms $\CO^1$ to be the intersection of $\CO$ with the closed $C^1$-ball of radius $\di$ centered at $\FI$.

We will need three auxiliary statements to prove Theorem~\ref{t:thm}.

\begin{prop}
The space $\CI$ is a nonempty complete metric space.
\end{prop}
\begin{proof}
$\CI$ is a closed subspace of a complete metric space $\mathrm{Diff}^1(\TT)$, therefore it is complete. It is nonempty because it contains $\FI$.
\end{proof}

\noindent Given $F \in \CO$, denote by $B(F) = \cup_{n\ge 0} F^{-n}(S)$ the basin of attraction of the semi-thick horseshoe~$\HS$. For $\eps > 0$ define the set $A_\eps \subset \CO$ as the set of all $F \in \CO$ such that $\leb(B(F)) > 1-\eps$, where $\leb$ is the Lebesgue measure on~$\TT$. Also, let $A_\eps^1 = A_\eps \cap \CI$.

\begin{lem} \label{l:open}
For any $\eps>0$ the set $A_\eps$ is open in the $\Homeo$-topology on $\CO$.
\end{lem}

\begin{lem} \label{l:dense}
For any $\eps>0$ the set $A^1_\eps$ is dense in $\CI$ (in the $C^1$-topology).
\end{lem}

We will prove these two lemmas below in sections~\ref{s:open} and~\ref{s:dense} respectively. Let us deduce Theorem~\ref{t:thm} from them.

\begin{proof}[Proof of Theorem~\ref{t:thm}]
Consider the set $A^1_0 := \bigcap_{n = 1}^\infty A^1_{1/n}$. By Lemmas~\ref{l:open} and~\ref{l:dense} the set $A^1_0$ is a residual subset of $\CI$. By the Baire theorem it is non-empty. Let us take as $\FH$ any diffeomorphism in~$A^1_0$. It follows from the definition of $A^1_0$ that $\leb(B(\FH)) = 1$.

Let us show that $\FH$ satisfies the conclusion of Theorem~\ref{t:thm}. The physical measure $\nu$ is obtained as the image of the $(1/2, 1/2)$-Bernoulli measure under the symbolic encoding map $\chi: \{0, 1\}^{\mathbb Z} \to \HS$ for the horseshoe. Obviously, the support of $\nu$ is $\HS$. Denote the basin of the measure~$\nu$ by $Bas$. Let us prove that $\nu(Bas)=1$, i.e., that for a $\nu$-a.e. point $x \in \HS$ the sequence $\delta^n_x$ will $\ast$-weakly converge to~$\nu$. This statement is equivalent to the analogous statement for the Bernoulli shift and the $(1/2, 1/2)$-Bernoulli measure. That latter statement follows from the Birkhoff theorem. 

Let $\leb_S = \leb|_S/\leb(S)$ be the probability Lebesgue measure on $S$ and $\pi: S \to \CT$ be the projection onto the vertical axis. From the construction of $\CT$ (see Proposition~\ref{p:equal_measures} in section~\ref{s:FB}) it follows that $\pi_* \nu = \pi_* \leb_S$. 
If the set $Bas \cap S$ intersects some horizontal fiber of $S$, it contains this whole fiber, because the distance between the images of any two points on this fiber tends to zero when we iterate the map. Therefore, 
\[
\leb_S(Bas \cap S) = (\pi_* \leb_S)(\pi(Bas \cap S)) = (\pi_* \nu)(\pi(Bas \cap S)) = \nu(Bas \cap S) = 1. 
\]

We showed that the basin of $\nu$ contains almost all points of~$S$. Iterating backwards, we conclude that the basin of $\nu$ contains almost every point of $\FH^{-n}(S)$ for any~$n$, and therefore almost every point of~$B(\FH)$. 
Since $\leb(B(\FH)) = 1$, this means that that the basin of~$\nu$ has full Lebesgue measure.

Let us prove that $\omega(x) = \HS$ for almost any $x$. Indeed, for any $x$ in the basin of $\nu$ we have $\omega(x) \supset \HS$, while for any $x$ in $B(\FH)$ we have $\omega(x) \subset \HS$.
\end{proof}

%% file: openness.tex
\section{\texorpdfstring{$A_\eps$}{A-eps} is open} \label{s:open}
In this section we will prove lemma~\ref{l:open}.

\begin{prop}\label{prop:ineqlip}
Suppose that $F\colon \TT\to\TT$ is a Lipschitz map. Then for any Borel subset $A\subset \TT$ one has $\mu(F(A)) \le \mathrm{Lip}^2(F)\cdot\mu(A)$. 
\end{prop}

\begin{proof}
For any Borel set $B\subset\TT$ the diameter of $F(B)$ is at most $\mathrm{Lip}(F)\cdot\mathrm{diam}(B)$.
Application of this inequality to the elements of any cover (from the definition of Hausdorff measure) of an arbitrary Borel set $A\subset\TT$ yields that for the two-dimensional Hausdorff measure $H_2$ the following inequality holds:
$$H_2(F(A)) \le \mathrm{Lip}^2(F)\cdot H_2(A).$$
Now it suffices to observe that the Lebesgue measure $\mu$ differs from the Hausdorff measure $H_2$ by a constant factor.  
\end{proof}

\begin{rem}\label{rem1}
For any $\hat \delta > 0$ the set $S$ can be covered by a union $\hat S = \hat S(\hat \delta)$ of finitely many rectangles in such a way that $\mu(\hat S\setminus S) < \hat \delta$. Indeed, $S$ is a Cartesian product of a ``vertical'' thick Cantor set $C_{thick}$ and a ``horizontal'' segment.  A finite cover of $C_{thick}$ by segments is obtained when we take the convex hull of $C_{thick}$ and remove finitely many intervals of its complement starting from the largest ones. Then we can obtain the required cover of $S$ by multiplying this cover of $C_{thick}$ by a horizontal segment.
\end{rem}

Consider an arbitrary homeomorphism $F\in A_\eps$. By the definition of $A_\eps$, we have ${\mu(B(F)) > 1-\eps}$. To prove Lemma~\ref{l:open}, we will show that if $G\in\CO$ is sufficiently close to $F$ in the $\Homeo$-metric, then $\mu(B(G)) > 1-\eps$.

Let $\delta = \mu(B(F)) - (1-\eps) > 0$.
Recall that $B(F) = \cup_{j = 0}^\infty F^{-j}(S)$.
The sets $F^{-j}(S)$ form an ascending sequence (by inclusion), therefore there is~$N\in\NN$ such that
\begin{equation}\label{ineq:1}
\mu\left(F^{-N}(S)\right) > \mu(B(F))-\delta/3.
\end{equation}

It suffices to prove that if $F$ and $G$ are sufficiently close to each other, then 
\begin{equation}\label{ineq:neweq}
\mu\left(G^{-N}(S)\right) > \mu\left(F^{-N}(S)\right) - 2\delta/3.
\end{equation}

Indeed, with \eqref{ineq:1} and \eqref{ineq:neweq} we could write the following estimate:
$$\mu(B(G)) \; \ge \; \mu(G^{-N}(S)) \; \overset{(\ref{ineq:neweq})}{>} \; \mu(F^{-N}(S)) - 2\delta/3 \; \overset{(\ref{ineq:1})}{>} \; \mu(B(F))-\delta \; = \; 1 -\eps,$$
which yields $G \in A_\eps$.

To prove inequality~\eqref{ineq:neweq}, we will need several auxiliary propositions.

\begin{prop} \label{p:open-approx}

For a fixed $N$ there is $\hat \delta > 0$ such that for any $G\in\CO$ one has

\[
 \mu\left(G^{-N}(S)\right) > \mu\left(G^{-N}({\hat S}(\hat \delta))\right) - \delta/3.
\]
\end{prop}
\begin{proof}
Recall that $S \subset \hat S$ (see Remark~\ref{rem1}) and that the homeomorphism $G$ is bi-Lipschitz with a constant $L$. Proposition~\ref{prop:ineqlip} yields the following inequality:
$$\mu(G^{-N}(\hat{S}\setminus S)) \le L^{2N}\cdot\mu(\hat{S}\setminus S) \le L^{2N}\cdot \hat \delta.$$

It suffices to take $\hat\delta < \frac{\delta}{3L^{2N}}$.
\end{proof}

\begin{prop} \label{p:open-open}
Let $F\in\CO$, $\hat S$ and $N$ be fixed. Then for any homeomorphism $G\in\CO$ that is sufficiently close to $F$ in the $\Homeo$-metric the following inequality holds
\[
\mu\left(G^{-N}(\hat S)\right) > \mu\left(F^{-N}(\hat S)\right)-\delta/3.
\]
\end{prop}

\begin{proof}

Let $\Pi$ be an arbitrary rectangle of the cover $\hat S$. Consider the set $F^{N}\circ G^{-N}(\Pi)$. This set is a topological disk whose boundary is close to the one of the rectangle $\Pi$ if $G$ is close to $F$. Therefore the measure of the symmetric difference $(F^{N}\circ G^{-N}(\Pi))\; \triangle\; \Pi$ is less than an arbitrary $\eps_1>0$, provided that $F$ and $G$ are sufficiently close. Application of Proposition~\ref{prop:ineqlip} yields 
\[
\mu\big(G^{-N}(\Pi)\; \triangle \;  F^{-N}(\Pi)\big) = \mu \big(F^{-N}\big(F^N\circ G^{-N}(\Pi) \; \triangle \; \Pi\big) \big) \le L^{2N}\cdot \eps_1.
\]
%Since for any measurable sets $A, B$ we have $\mu(A) \ge \mu(B) - \mu(A \triangle B)$,
Therefore
\[
\mu(G^{-N}(\Pi)) \ge \mu(F^{-N}(\Pi)) - L^{2N}\cdot \eps_1.
\]
So, if $F$ and $G$ are sufficiently close, this inequality holds for any rectangle of the cover $\hat S$. Denote the number of rectangles in the cover $\hat S$ by $P$. Then we have
$$\mu(G^{-N}(\hat S)) \ge \mu(F^{-N}(\hat S)) - P \cdot L^{2N} \cdot \eps_1.$$
If $\eps_1$ is small enough, we have obtained the required inequality.
\end{proof}

To sum up, if $F$ and $G$ are sufficiently close to each other, we have:

$$\mu\left(G^{-N}(S)\right) \overset{\mbox{П.}\ref{p:open-approx}}{>} \mu\left(G^{-N}(\hat S)\right) - \delta/3 \overset{\mbox{П.}\ref{p:open-open}}{>} \mu\left(F^{-N}(\hat S)\right) - 2\delta/3 > \mu\left(F^{-N}(S)\right) - 2\delta/3.$$

Therefore, we have obtained inequality~\eqref{ineq:neweq}, which finishes the proof of Lemma~\ref{l:open}.

%% file: density.tex
\section{\texorpdfstring{$A^1_\eps$}{A-1-eps} is dense} \label{s:dense}
\subsection{Segments of level~\texorpdfstring{$n$}{n}}
This section is devoted to the proof of Lemma~\ref{l:dense}. We begin with the following important definition.
\begin{defin} \label{d:segments}
For $F \in \CO$  \emph{segments of level $0$} are the connected components of the intersections of~$\DD$ with the vertical fibers. \emph{Segments of level~$n$} are their $n$-th preimages. 
\end{defin}
Let us give the plan of the proof of Lemma~\ref{l:dense} first. Given a map $F_0 \in \CO^1$, we want to obtain a map $F \in A_\eps \cap \CO^1$ by a small perturbation. 
First we perform the vertical linearization (see. section~\ref{s:lin} below) and obtain a map $\FPL \in \CO$ which is linear in restriction to all segments of levels greater than $N$ for some fixed large~$N$. In section~\ref{s:density-point} we use a well-known argument with the Lebesgue density point theorem combined with distortion control to show that for this new map $\Leb(B(\FPL)) = 1$. Thus $\FPL \in A_\eps$. Unfortunately, the map $\FPL$ is not in the class $\CO^1$ anymore. Therefore in section~\ref{s:smoothing} we smooth this map via a second perturbation to obtain the required~$F$. This smoothing perturbation can be made arbitrarily small in the $\Homeo$-topology, which thanks to openness of $A_\eps$ yields that $F \in A_\eps$. Finally, we check that $F$ turns out to be $C^1$-close to $F_0$ provided that $N$ is large.

\subsection{The density point argument} \label{s:density-point}
In this section we prove the following lemma.
\begin{lem} \label{l:N-linear}
Suppose that the map $H \in \CO$ is linear in restriction to all $H$-segments of level greater than $N$ for some $N \in \NN$. Then $Leb(B(H)) = 1$. 
\end{lem}
\noindent For this end we will need two auxiliary statements.
\begin{prop} \label{p:toD}
For any map $H \in \CO$ and any point $x \in \TT \setminus S$ there exists an integer $m \ge 0$ such that $H^m(x) \in \DD$.
\end{prop}
\begin{proof}
If $x\in\DD$, there is nothing to prove. Suppose that $x \in \UK$. Then $x$ lies in some connected component $\Pi$ of the set $\UK \setminus S$. This connected component is a horizontal stripe. Let $\Pi_0$ be the middle stripe of $\UK \setminus S$. The dynamics on the horseshoe $\HS$ is the same as for the standard Smale horseshoe, so there exists $n \ge 0$ such that $\FI^n(\Pi) \subset \Pi_0$. The stripe $\Pi$ is bounded by two horizontal segments of $\dh \Pi$ and two vertical segments of~$\dv \Pi$. Since $\dh \Pi \subset S$, it follows from the property~\ref{i:S} of the class~$\CO$ that $H$ and $\FI$ coincide on~$\dh \Pi$. Since the vertical foliation is invariant for both maps, we have $H^n(\Pi) = \FI^n(\Pi) \subset \Pi_0$.

By condition~\ref{i:to-D} from the definition of~$\CO$ we have $H(\Pi_0) \subset \DD$. Therefore, $H^{n+1}(\Pi) \subset \DD$ and, in particular, $H^{n+1}(x) \in \DD$. 
\end{proof}

\begin{cor}\label{c:SorD}
For any $x\in\TT$ and any $H\in\CO$, if the positive orbit of~$x$ does not end up in~$S$, then it visits~$\DD$ infinitely many times.
\end{cor}

\begin{proof}[Proof of Lemma~\ref{l:N-linear}]
Assume the contrary. Then the set $\Dop = \TT\setminus B(H)$ has positive measure and, therefore, its intersection with some vertical segment has a density point with respect to the Lebesgue measure on this segment. Denote this point by $x$. Recall that $\TT \setminus \UK$ is the union of vertical $H$-segments of level 0. By Corollary~\ref{c:SorD} the positive orbit of $x$ visits $\DD$ infinitely many times. Hence, $x$ belongs to $H$-segments of arbitrary large levels. Since $H^{-1}$ contracts the vertical segments (property~\ref{i:expand} of the class~$\CO$), the lengths of these segments tend to zero, because the lengths of segments of level 0 are uniformly bounded. Since $x$ is a density point, the density of the set $\Dop$ in these segments is arbitrarily close to one. Let us deduce a contradiction from this.

Since $H$ is, by assumption, linear in restriction to $H$-segments of large levels, our segment with high density of $\Dop$ linearly expands to a segment of level~$N$. Then it takes $N$ iterations to non-linearly expand it to a segment of level 0. Since the distortion during these $N$ iterations is bounded (this follows from properties~\ref{i:expand} and~\ref{i:bilip} of the class~$\CO$), we may assume that inside the segment of level 0 we still have density of $\Dop$ close to one. Finally, there exists $M\in\NN$ such that for any segment $I$ of level 0 the image $H^M(I)$ cuts all the way trough $\intt \UK$. Such $M$ exists because each fiber of the vertical foliation is an irrational winding of $\TT$ and $H$ uniformly expands the segments of vertical fibers. Then inside $H^M(I)$ the relative measure of the points that belong to $S$ is at least $\frac{\mu(S)}{|H^M(I)|}$, where $\mu(S)$ is the one-dimensional Lebesgue measure of the projection of $S$ to a vertical fiber. Since  the distortion during these $M$ iterations is bounded, the fraction of points of $\Dop$ inside $H^M(I)$ can be arbitrarily close to one. Since $\Dop\cap S = \emptyset$, this leads to a contradiction.
\end{proof}

\subsection{Linearization} \label{s:lin}
In this section for given $F_0 \in \CI$ and (large) $N \in \NN$ we will construct a homeomorphism $\FPL \in \CO$ that is linear in restriction to every $\FPL$-segment of level greater than~$N$ and coincides with $F_0$ outside the union of these segments.

Consider $H \in \CO$ and let $J$ be an $H$-segment of level~$n$. We define \emph{the linearization procedure} on $J$ as follows: the restriction $H|_J$ is replaced by an affine map such that the image of $J$ stays the same.
To obtain the map $\FPL$, we first linearize $F_0$ on all $F_0$-segments of level~$N+1$. This yields a map $F_{N+1}\in\CO$ that is linear on all $F_{N+1}$-segments of level~$N+1$. Then we linearize $F_{N+1}$ on all $F_{N+1}$-segments of level~$N+2$ and obtain a map $F_{N+2}$ linear on all $F_{N+2}$-segments of level~$N+2$. We will prove that $F_{N+2}$ is also linear on all $F_{N+2}$-segments of level~$N+1$. Continuing this process, we will obtain a sequence of maps and then show that it $C^0$-converges to the required map~$\FPL$.

\subsubsection{Stripes of level \texorpdfstring{$n$}{n}}
Consider a homeomorphism $H \in \CO$ and an integer $n>L$. Recall that $H$-segments of level~$0$ are the connected components of vertical fibers inside $\TT \setminus \UK$, whereas segments of level~$n$ are the $n$-th preimages of segments of level zero. Thus, segments of level $n$ are the connected components of vertical fibers inside $\TT \setminus H^{-n}(\UK)$. Denote $W_m := \dh H^{-m}(\UK)$. Then we have $W_m = H^{-m}(\dh UK)$, where $\dh UK$ consists of two ``local stable manifolds'' of the two fixed points $p_0$ and $p_1$ of our horseshoe. The set $W_m$ consists of two curves $W_m^0$ and $W_m^1$ which are two longer pieces of the ``stable manifolds'' of $p_0$ and $p_1$. It is easy to check that the curves $W_m^0$ and $W_m^1$ locally are the graphs of (not necessarily smooth) functions from the horizontal axis to the vertical one. Indeed, this is obvious for $W_0$ and is checked inductively for larger~$m$, using that the vertical foliation is preserved. Since $H(\dh \UK) \subset \dh \UK$, for any $m \ge 0$ we have 
\[
W_{m+1} \supset W_m.
\]
Since $n>L$, it follows from property~\ref{i:dh} of class~$\CO$ that 
\begin{equation} \label{e:h}
\dh R \subset W_n.
\end{equation}
Recall also property~\ref{i:v-vodu}:
\begin{equation} \label{e:v}
\dv H^{-n}(\UK) \cap R = \emptyset.
\end{equation}

Since outside $R$ the map $H$ is already linear (property~\ref{i:lin}), we are interested only in segments of level~$n$ that intersect $\intt R$. It follows from~\eqref{e:h} that such segments are included in $R$. Properties~\eqref{e:h} and~\eqref{e:v} imply that the set $W_n$ splits $R$ into stripes that are more or less horizontal. Let us call those stripes that are included in $F^{-n}(\DD)$ \emph{stripes of level~$n$}.
All this gives us the following statement.
\begin{prop} \label{p:stripes}
For any $n>L$ and $H \in \CO$ the union of all segments of level~$n$ that intersect~$\intt R$ equals the union of all stripes of level~$n$.
\end{prop}

\subsubsection{Properties of the stripes of level~\texorpdfstring{$n$}{n}}

\begin{prop} \label{p:Markov}
Let $\Pi_l$ be a stripe of level $l>L$ and $\Pi_k$ be a stripe of level~$k>l$. Then either $\Pi_k \subset \Pi_l$ or the closures of $\Pi_k$ and $\Pi_l$ are disjoint. 
\end{prop}
\begin{proof}
Suppose that $\Pi_k \not \subset \Pi_l$. Let us show first that the interiors of the two stripes are disjoint. Suppose the contrary: $\intt \Pi_k \cap \intt \Pi_l \ne \emptyset$. Then $\dh \Pi_l$ intersects $\intt \Pi_k$. But $\dh \Pi_l \subset W_l \subset W_k$. Hence $W_k$ intersects $\intt \Pi_k$, which contradicts the definition of the stripes of level~$k$.

Now, if the boundaries of the stripes intersect, then, since $\Pi_k \not \subset \Pi_l$, the upper boundary of $\Pi_k$ intersects the lower boundary of $\Pi_l$ (or vice versa, and in this case the argument is analogous). Let assume that $p_0$ is the upper one of the two fixed points of the horseshoe; then the aforementioned boundaries are included into the curves $W_k^0$ and $W_l^1$, respectively. But these curves do not intersect, because they are the pieces of the stable sets of the points $p_0$ and $p_1$. This contradiction finishes the proof.  
\end{proof}\noindent Given $H \in \CO$ and $n>L$, denote by $\Lcal_n(H)$ the map obtained from $H$ by linearization on all segments of level~$n$.
\begin{prop} \label{p:stripes-preserved}
Let $k>L$ and $H \in \CO$. Then for every $l \in (L, k]$ the maps $H$ and $\Lcal_k(H)$ have the same stripes of level~$l$, and $W_k(\Lcal_k(H)) = W_k(H)$.
\end{prop}
\begin{proof}
Since the interiors of the stripes of level $k$ do not cross $W_k(H)$, we have $H|_{W_k(H)} = \Lcal_k(H)|_{W_k(H)}$. So, if we apply $\Lcal_k(H)$ to $W_k(H)$ $k$ times, we get the horizontal boundary of $UK$. Then we must have that $W_k(H)=W_k(\Lcal_k(H))$. %Since the set $W_k(H)$ consists of two pieces of the local ``stable manifolds'' of the fixed points of the horseshoe, and therefore is not changed during the linearization procedure, $W_k(H)=W_k(\Lcal_k(H))$.
Since $W_l(H) \subset W_k(H)$, we have $W_l(H) = W_l(\Lcal_k(H))$. This and preservation of the vertical foliation implies that $H^{-l}(\UK)=\Lcal_k(H)^{-l}(\UK)$. Thus $\Lcal_k(H)$ has the same stripes of level~$l$ as~$H$.  
\end{proof}

\subsubsection{Roughly horizontal stripes}
We will show that for the maps $F_i$, introduced in the beginning of section~\ref{s:lin}, all $F_i$-stripes will be \emph{roughly horizontal} in the following sense. Consider the coordinates $(x, y)$ on $R$: the axis $y$ is vertical, and the axis $x$ is horizontal. Denote by $J_h$ and $J_v$ the projections of $R$ to the horizontal and the vertical axes.
\begin{defin}
 Consider two $C^1$-smooth functions $\varphi_1, \varphi_2 \colon J_h \to J_v$ such that for every $x\in J_h$
\begin{itemize}
\item $(x, \varphi_1(x)) \in R$, $(x, \varphi_2(x)) \in R$;
\item $|\varphi_1'(x)|<1$, $|\varphi_2'(x)|<1$;
\item $\varphi_1(x) < \varphi_2(x)$. 
\end{itemize}
\emph{A roughly horizontal stripe} is a set of all points of $R$ that lie between the graphs of such two functions. The segments $x \times [\varphi_1(x), \varphi_2(x)]$ are called \emph{the vertical segments} of this stripe.
\end{defin}
\noindent  Consider a roughly horizontal stripe~$\Pi$. Given two smooth maps $G, H: \Pi \to \TT$ we denote
\[
\dist_\Pi(G, H) = \sup_{x \in \Pi} \|dG(x) - dH(x)\|.
\]
For $H \in \CO$ denote by $\Lcal_\Pi(H): \Pi \to H(\Pi)$ the map obtained from $H$ after linearization on all vertical segments of the stripe~$\Pi$. It is easy to see that $\Lcal_\Pi(H)$ is $C^1$-smooth, since the horizontal boundary of the stripe is defined by $C^1$-smooth functions. For a roughly horizontal stripe $\Pi$ and $H \in \CI$, we can estimate how close $\Lcal_\Pi(H)$ is to~$H$, as the following lemma claims.
\begin{lem}\label{l:delta}
Consider an arbitrary diffeomorphism $F_0 \in \CI$ and any roughly horizontal stripe~$\Pi$. Suppose that, for some $\delta > 0$, for any two points $p, q \in \Pi$ on the same vertical fiber one has
$\|dF_0(p) - dF_0(q)\| < \delta.$
Then $\dist_\Pi(F_0, \Lcal_\Pi(F_0)) < \sqrt{5}\delta$.
\end{lem}
\noindent This lemma will be proved in section~\ref{s:technical}.

\begin{cor}\label{cor:delta}
For any $F_0 \in \CI$ and $\delta>0$, for any sufficiently thin roughly horizontal stripe~$\Pi$ (in particular, for an $H$-stripe of sufficiently large level for arbitrary $H\in\CO$) the following holds:
$\dist_{C^1}(F_0|_\Pi, \Lcal_\Pi(F_0)) < \delta$.
\end{cor}
\begin{proof}
Since the stripe is thin, the derivative $dF_0$ has small oscillation on its vertical segments, because of the uniform continuity, so the previous lemma is applicable (with $\delta$ replaced with $\delta/2\sqrt{5}$). The lemma yields that $dF_0$ and $d\Lcal_\Pi(F_0)$ are $\delta/2$-close. Moreover, the $C^0$-closeness follows automatically, because the stripe is thin. Finally, we have the required $C^1$-closeness.
\end{proof}

\subsubsection{Induction} \label{s:induction}
In the beginning of section~\ref{s:lin} for a given $F_0 \in \CI$ and $N > L$ we defined the sequence $(F_i)_{i=N}^\infty$ in the following way: $F_N = F_0$, and $F_{i+1}=\Lcal_{i+1}(F_i)$ for $i \ge N$. For $k > N$ and $H \in \CO$, let us call the $H$-stripe of level $k$ \emph{dependent} if it is included into another $H$-stripe of level in range from $N+1$ to $k-1$. Let us call the stripe \emph{independent} otherwise. By Proposition~\ref{p:Markov}, the interiors of different independent stripes are disjoint.
\begin{prop} \label{p:induction}
For any $F_0 \in \CI$ such that $\dist_{C^1}(F_0, \FI) < \di$ there exists a number~$N_0 \in \NN$ such that for any $N>N_0$ the following holds for the sequence $(F_i)_{i=N}^\infty$ constructed above. For any $i \ge N$:
\begin{enumerate}
\item \label{ii:CO} $F_i \in \CO$; 
\item \label{ii:coincide} The maps $F_i$ and $F_{i-1}$ have the same stripes of levels from $N+1$ to $i$;
\item \label{ii:Pi-j} $F_i = F_0$ outside the (finite) union $\bigsqcup_j \Pi_j$ of all independent $F_i$-stripes of levels from $N+1$ to $i$; 
\item \label{ii:Lcal} For any $j$ we have $F_i|_{\Pi_j} = \Lcal_{\Pi_j}(F_0)$; 
\item \label{ii:close} For any $j$ the restriction of $F_i$ to $\Pi_j$ is smooth and $\di$-close to $\FI$ in the sense of the $\dist_{\Pi_j}$-metric;
\item \label{ii:roughly} For any $m \ge 0$ the curves $W_m^0(F_i)$ and $W_m^1(F_i)$ are smooth, and their tangent vectors are contained in the cones~$C_H$ (therefore, for every $m>N$ all $F_i$-stripes of level~$m$ are roughly horizontal).
\end{enumerate}
\end{prop}

\begin{proof}
Choose $N_0>L$ such that for any $n>N_0$, for our $F_0$ and any map $H \in \CO$, for any $H$-stripe $\Pi$ of level $n$ one would have
\begin{equation} \label{e:close}
\dist_\Pi(\Lcal_\Pi(F_0), \FI) < \di.
\end{equation}
%\todo[inline]{Переписал абзац ниже}
Let us prove that this is possible. Since $\dist_{C^1}(F_0, \FI) < \di$, we have $\dist_\Pi(F_0, \FI) < \di$. Define the \emph{width} of the stripe as the maximal length of its vertical segments. Choose $h$ so that for any roughly horizontal stripe $\Pi$ of width at most $h$ we had $\dist_\Pi(\Lcal_\Pi(F_0), F_0) < \di - \dist_\Pi(F_0, \FI)$. Then, by the triangle inequality, for any such stripe we have~\eqref{e:close}. Now, let $N_0$ be so large that for any $n>N_0$ the width of any level $n$ stripe $\Pi$ of any map $H \in \CO$ be less than~$h$. This is possible by property~\ref{i:expand} of the class~$\CO$.   

Now we will prove the required properties using induction in~$i$. 
For $i=N$ we have $F_N = F_0$, so there are no stripes $\Pi_j$ at all. Thus, we only need to prove the statement about the cones in~\ref{ii:roughly}. The curves $W_0^{0, 1}$ are strictly horizontal. Since $dF_0^{-1}(C_H) \subset C_H$, the tangent vectors to the curves $W_m^{0, 1}  = F_0^{-m}(W_0^{0, 1})$ lie in the cones~$C_H$.

Assume that for some $i \ge N$ all the properties are established. Let us prove that they hold for~$i+1$. The map $F_{i+1}$ is obtained from $F_i$ by linearization on all $F_i$-stripes of level~$i+1$. Since on all dependent stripes the restriction of $F_i$ to any vertical segment is already linear, linearization does not change the map on these stripes. Therefore \emph{$F_{i+1}$ is obtained from $F_i$ by linearization on all independent $F_i$-stripes of level~$i+1$.} Since on the vertical boundaries of these stripes the map is already linear by property~\ref{i:lin} of the class $\CO$ (because these boundaries are in $\dv R$), linearization does not create discontinuities on the vertical boundaries of the stripes, and so the map $F_{i+1}$ will be a homeomorphism.

It follows from the choice of~$N_0$ (see inequality~\eqref{e:close}) that on any independent $F_i$-stripe of level~$i+1$ the map $F_{i+1}$ is $\di$-close to $\FI$ in the $\dist_\Pi$-metric. Therefore property~\eqref{e:di} (from section~\ref{s:Baire}) implies that inside every such stripe the homeomorphism $F_{i+1}$ is $L$-bi-Lipschitz. Outside the union of those stripes the map $F_{i+1}$ coincides with $F_i$ and is bi-Lipschitz by the induction assumption. The triangle inequality allows to deduce now that $F_{i+1}$ is bi-Lipschitz on the whole $\TT$. Note that in this argument it is irrelevant whether we include the boundary of the stripes under consideration into the stripes themselves or into their complement.
 
Let us check that $F_{i+1}$ satisfies the rest of the properties of the class~$\CO$ (provided that $F_i$ does). 

\begin{enumerate}
\item $F_{i+1}|_S=\FI|_S$.

Note that (open) $F_i$-stripes of level~$i+1$ do not intersect $S$. Indeed, $F_i(S) \subset S$, so $S \subset F_i^{-(i+1)}(S)$, but since $S\subset\UK$, we have $S\subset F_i^{-(i+1)}(S) \subset F_i^{-(i+1)}(\UK)$, but the stripes are in the complement of the latter set by definition. Thus $F_{i+1}|_S = F_i|_S = \FI|_S$.

\item $F_{i+1}$ preserves the vertical foliation. 

This follows from the fact that our linearization on the vertical fibers does not change them.

\item $H$ is $L$-bi-Lipschitz.

We have already proved that.

\item $F_{i+1} = F_{Lin}$ outside $R$.

Since all stripes are inside $R$, outside $R$ we have $F_{i+1}=F_i=\FL$.

\item The restriction of $F_{i+1}$ to any vertical fiber dilates by a factor at least $1.1$ and at most $L$.

When we linearize $F_i$ on a vertical segment $J$, the map is replaced by the linear one which dilates $J$ so that the ratio of the length of $J$ and its image is the same as for $F_i$. Therefore the maximal dilation cannot become stronger and the minimal cannot become weaker. 

\item $F_{i+1}(\intt \UK) \cap \intt \UK = \intt UV_0 \cup \intt UV_1$.

Since $\dh \UK = W_0 \subset W_{i+1}$, the set $\dh \UK$ does not intersect the interiors of the $F_i$-stripes of level~$i+1$. Therefore, $F_{i+1}(\dh \UK) = F_i(\dh \UK)$. Since $\UK$ consists of the segments of vertical fibers in-between its upper and lower boundaries and the images of these vertical fibers are vertical fibers as well, we get $F_{i+1}(\UK) = F_i(\UK)$. Therefore, $F_{i+1}(\intt \UK) \cap \intt \UK = F_i(\intt \UK) \cap \intt \UK = \intt UV_0 \cup \intt UV_1$.
The last equality holds by the induction assumption.

\item $F_{i+1}^{-L} = \FI^{-L}$ on $\dh \UK$.

$F_i^{-L}(\dh \UK) = W_L(F_i) \subset W_{i+1}(F_i).$ This subset does not intersect the interior of the $F_i$-stripes of level~$i+1$, so $F_{i+1}=F_i$ on $W_L(F_i)$. Since $W_L$ is forward-invariant, $F_{i+1}^L=F_i^L$ in restriction to this set. Thus, $F^L_{i+1}(W_L) = \dh \UK$, and $F^{-L}_{i+1}=F^{-L}_i=\FI^{-L}$ on $\dh \UK$.

\end{enumerate}

So, we proved that $F_{i+1} \in \CO$. Let us prove the rest of the claims.
\begin{enumerate}
\setcounter{enumi}{1}

\item The maps $F_i$ and $F_{i+1}$ have the same stripes of levels from $N+1$ to $i+1$.

This follows from Proposition~\ref{p:stripes-preserved}.

\item $F_{i+1} = F_0$ outside the (finite) union $\bigsqcup_j \Pi_j$ of all independent $F_i$-stripes of levels from $N+1$ to~$i$; 

By the induction assumption, outside the union of the independent stripes of levels from $N+1$ to $i$ we have $F_i=F_0$. Since $F_{i+1}$ is obtained by linearizing $F_i$ on all independent $F_i$-stripes of level~$i+1$, we have the required statement.

\item For any $j$ one has $F_{i+1}|_{\Pi_j} = \Lcal_{\Pi_j}(F_0)$; 

For the stripes of levels from $N+1$ to $i$ this follows from the induction assumption. If $\Pi_j$ is an independent stripe of level $i+1$, then $F_i|_{\Pi_j} = F_0|_{\Pi_j}$, so $F_{i+1}|_{\Pi_j} = \Lcal_{\Pi_j}(F_i) = \Lcal_{\Pi_j}(F_0)$.

\item For any $j$ the restriction of $F_{i+1}$ to $\Pi_j$ is smooth and $\di$-close to $\FI$ in the sense of the $\dist_{\Pi_j}$-metric; 

This follows from the previous claim (claim~\ref{ii:Lcal}) and inequality~\eqref{e:close}.

\item For any $m \ge 0$ the curves $W_m^0(F_i)$ and $W_m^1(F_i)$ are smooth and their tangent vectors always lie in the cones $C_H$ (therefore, for any $m>L$ all $F_i$-stripes of level~$m$ are roughly horizontal).

For arbitrary $j$, let us write $W_j$ and $W^{0,1}_j$ instead of $W_j(F_{i+1})$ and $W^{0,1}_j(F_{i+1})$.

Since $W_{i+1}$ contains the horizontal boundaries of all the stripes $\Pi_j$, the map $F_{i+1}$ is smooth on $\TT \setminus W_{i+1}$ (and also at the four endpoints of the curves $W^0_{i+1}$ and $W^1_{i+1}$). 

By Proposition~\ref{p:stripes-preserved}, $W^{0, 1}_{i+1}(F_{i+1})=W^{0, 1}_{i+1}(F_i)$. By the induction assumption, these curves are smooth and their tangent vectors are in the cones $C_H$.

It is clear that 
\begin{equation} \label{e:grow-W}
W^0_{i+2} \setminus W^0_{i+1} = F_{i+1}^{-1}(W^0_{i+1} \setminus W^0_i).
\end{equation}
Since the map $F_{i+1}$ is smooth on $W^0_{i+2} \setminus W^0_{i+1}$, the map $F_{i+1}^{-1}$ is smooth on $W^0_{i+1} \setminus W^0_i$. Thus~\eqref{e:grow-W} yields that $W^0_{i+2} \setminus W^0_{i+1}$ is the union of two smooth curves. Since $F_{i+1}$ is smooth at the endpoints $a$ and $b$ of the curve $W^0_{i+1}$ (which are in the boundary of $W^0_{i+2} \setminus W^0_{i+1}$), we can establish the smoothness of the curves obtained from $W^0_{i+2} \setminus W^0_{i+1}$ by adding two small pieces of $W^0_{i+1}$ near the endpoints $a$ and $b$. Therefore, $W^0_{i+2} = (W^0_{i+2} \setminus W^0_{i+1}) \cup W^0_{i+1}$ is a smooth curve. 

By claims~\ref{ii:Pi-j} and~\ref{ii:close}, at any point $x \in \TT \setminus W^0_{i+1}$ the derivative $d_x F_{i+1}$ is $\di$-close to $d_x \FI$. By~\eqref{e:di} we have $d_x F_{i+1}(C_H(x)) \supset C_H(F_{i+1}(x))$. Therefore, at any point $y \in W^0_{i+1} \setminus W^0_i$ we have $d_y F^{-1}_{i+1}(C_H(y)) \subset C_H(F_{i+1}^{-1}(y))$. Using~\eqref{e:grow-W} and the fact that the vectors tangent to $W^0_{i+1}$ are in the cones $C_H$, we obtain that the tangent vectors of $W^0_{i+2} \setminus W^0_{i+1}$ also lie in the cones~$C_H$. 

We proved both properties for the curve $W^0_{i+2}$, using only the same properties for  $W^0_{i+1}$. The transition from $i+2$ to $i+3$ etc. is performed analogously, as well as the argument for~$W^1_m$.
\end{enumerate}
\end{proof}

\subsubsection{Construction of \texorpdfstring{$\FPL$}{F-infty}}
Let us deduce from Proposition~\ref{p:induction} the following corollary.
\begin{cor} \label{c:induction}
For any $i > N$ the map $F_i$ is linear in restriction to any segment of level $N+1, \dots, i$. Moreover, for any $j>i$ the stripes of these levels are the same for~$F_i$ and $F_j$, and $F_j$ is obtained from $F_i$ by applying linearization to all independent $F_j$-stripes of levels from $i+1$ to~$j$.  
\end{cor}
\begin{proof}
By assertion~\ref{ii:Lcal} from Proposition~\ref{p:induction}, the map $F_i$ is linear in restriction to any vertical segment of any independent stripe of level $N+1, \dots, i$. Any dependent stripe lies inside some independent stripe, so $F_i$ is linear in restriction to its segments as well. By Proposition~\ref{p:stripes}, any segment of level $i$ is included either into a stripe of level~$i$ or into the complement of~$R$. In any case, $F_i$ is linear in restriction to this segment. 

The second statement follows directly from assertion~\ref{ii:coincide} of Proposition~\ref{p:induction} and the third follows from the first one and claims~\ref{ii:Pi-j} and~\ref{ii:Lcal} of the aforementioned proposition.

\end{proof}

\begin{prop} \label{p:FP}
For any $F_0 \in \CI$ such that $\dist_{C^1}(F_0, \FI) < \di$ there is an integer $N_0 > L$ such that for any $N>N_0$ one can construct a homeomorphism $\FPL$ such that
%is linear on all $\FPL$-segments of level greater than $N$.
\begin{enumerate}
\item \label{iii:CO} $\FPL \in \CO$; 
\item \label{iii:Pi-j} $\FPL=F_0$ outside the (countable) union $\bigsqcup_j \Pi_j$ of all independent $\FPL$-stripes of levels $N+1, N+2, \dots$; 
\item \label{iii:Lcal} For any $j$ we have $\FPL|_{\Pi_j} = \Lcal_{\Pi_j}(F_0)$; 
\item \label{iii:close} For any $j$ the restriction of $\FPL$ to $\Pi_j$ is smooth and $\di$-close to $\FI$ in the $\dist_{\Pi_j}$-metric; 
\item \label{iii:roughly} For any $m>L$ all $\FPL$-stripes of level $m$ are roughly horizontal;
\item \label{iii:linear} $\FPL$ is linear in restriction to any its segment of level greater than~$N$.
\end{enumerate}
\end{prop}

\begin{proof}
Take the number~$N_0$ from Proposition~\ref{p:induction}. For an arbitrary $N>N_0$, consider a sequence $(F_i)_{i=N}^\infty$ introduced in the beginning of section~\ref{s:induction}. 
Let us prove that this sequence is fundamental in the space $\CO$ endowed with the metric $$\dist_{\Homeo}(F,G) = \max(d_{C_0}(F,G), d_{C_0}(F^{-1}, G^{-1})).$$

Consider two homeomorphisms $F_k$ and $F_l$ with $l>k$. By Corollary~\ref{p:induction}, $F_k$ and $F_l$ coincide outside the union of all independent $F_l$-stripes of levels from $k+1$ to $l$. 
Since on the boundaries of these stripes we have $F_k=F_l=F_0$, we see that $\dist_{C^0}(F_k, F_l)$ is limited from above by the maximal vertical width of the $F_0$-images of these stripes.
Analogously, $\dist_{C^0}(F_k^{-1}, F_l^{-1})$ is bounded by the maximal width of the stripes themselves. If $k$ is large enough, all these stripes are very thin by property~\ref{i:expand} of the class~$\CO$. 
Thus, for any $\eps>0$ one can find an integer~$n$ such that for $l>k>n$ the homeomorphisms~$F_k$ and $F_l$ are $\eps$-close in our metric.  

The set $\CO$ is closed in $\Homeo(\TT)$, because every condition in its definition is a closed one. Therefore, $\CO$ is a complete metric space, as a closed subset of a complete metric space~$\Homeo(\TT)$. Thus, the sequence $(F_i)_{i=N}^\infty$ converges. Let $\FPL$ be its limit. Then $\FPL \in \CO$.

Properties $2-6$ that we need to establish follow from Proposition~\ref{p:induction} and Corollary~\ref{c:induction}. It should be noted that on the boundaries of the independent $\FPL$-stripes of levels greater than $N$ all maps $F_i$ coincide with $F_0$, and therefore the same holds for the map $\FPL$.  
\end{proof}

\subsection{Smoothing} \label{s:smoothing}
Now we can finally prove Lemma~\ref{l:dense}.
For a given $F_0 \in \CI$ and an arbitrary $\delta>0$ we will construct a diffeomorphism $F \in A_\eps \cap \CI$ such that $\dist_{C^1}(F, F_0) < C\delta$, where the constant $C$ is independent of~$F_0$. We can assume that $\dist_{C^1}(F_0, \FI) < \di$: if $\dist_{C^1}(F_0, \FI) = \di$, Lemma \ref{p:induction} cannot be applied, so we will replace $F_0$ by another map $\tilde F_0$ such that $\dist_{C^1}(\tilde F_0, \FI) < \di$, $\dist_{C^1}(F_0, \tilde F_0) < \delta/2$ and construct $F$ for the new $\tilde F_0$, with $\delta$ replaced with $\delta/2$.

Consider the number $N_0$ from Proposition~\ref{p:FP}. By Corollary~\ref{cor:delta}, we can take a large $N>N_0$ such that for the map $\FPL$ from Proposition~\ref{p:FP} the following holds: the restrictions of $F_0$ and $\FPL$ to any $\FPL$-stripe of level greater than $N$ are $\delta$-close in $C^1$.

\begin{comment}
\begin{itemize}
\item[] Для любой $\FPL$-полосы $\Pi$ уровня больше $N$ для любых двух точек $x,y \in \Pi$, лежащих на одном вертикальном отрезке,
\[\|dF_0(x) - dF_0(y)\| < \delta/\sqrt{5}.\]
\end{itemize}
Поэтому мы можем применить к этой полосе лемму~\ref{l:delta}. Подставляя в эту лемму $\delta/\sqrt{5}$ вместо $\delta$, получим, что ограничения $F_0$ и $\FPL$ на любую $\FPL$-полосу уровня больше $N$ будут $\delta$-близки в $C^1$. \todo{Зазор! Лемма говорит, что они близки в метрике Леши.}
%(note that they are $C^0$-close because the stripes are very thin when $N$ is large). И ЗАЧЕМ ЭТО ЗАКОММЕНТИРОВАНО, А? А?
\end{comment}

It suffices to prove that arbitrarily close to the map $\FPL$ in the $\Homeo$-topology there is a map $F \in \CI$ such that $\dist_{C^1}(F, F_0) < C\delta$. Then the proof of Lemma~\ref{l:dense} goes as follows. By Proposition~\ref{l:N-linear}, $\mu(B(\FPL))=1$. Thus, $\FPL \in A_\eps$. Since by Lemma~\ref{l:open} the set $A_\eps$ is open in $C^0$, whereas $F$ is $C^0$-close to $\FPL$, we have $F \in A_\eps$. Since $F$ is close to $F_0$ in~$C^1$, this finishes the proof of Lemma~\ref{l:dense}.

It remains to construct the map $F$, and for that matter we need the following lemma.	

\begin{lem}\label{l:incaps}
There exists a universal constant $C > 0$ such that the following holds. Take any $F_0\in\CI$ and any roughly horizontal stripe~$\Pi$. Then for any $\gamma > 0$ there exists an $F\in C^1(\TT, \TT)$ such that
\begin{itemize}
\item $F = F_0$ outside $\Pi$;
\item $\dist_{C^1}(F_0|_{\Pi}, F|_{\Pi}) \le C\cdot\dist_{C^1}(F_0|_{\Pi}, \FBL)$;
\item $\dist_{C^0}(\FBL, F|_{\Pi}) < \gamma$;
\item $F$ preserves the vertical foliation. 
\end{itemize} 
\end{lem}
The proof of this lemma will be given in section~\ref{s:technical}.
Now let us get back to the proof of lemma~\ref{l:dense}.
Let $\gamma > 0$ be so small that the $\gamma$-neighborhood of the map $\FPL$ in $\CO$ in the $\Homeo$-topology is included into~$A_\eps$. 
Since the width of the $\FPL$-stripes tends to zero as their level grows, there is an integer $N_1 > N$ such that outside the union of all $\FPL$-stripes of levels $N+1, \dots, N_1-1$ the maps $\FPL$ and $F_0$ are $\gamma$-close in the $\Homeo$-topology. Now, if we change the map $F_0$ in such a way that on the union of these stripes the new map be close to $\FPL$, then this new map will also lie in~$A_\eps$.

Let us modify the map $F_0$ in the following way: for every independent $\FPL$-stripe $\Pi$ of level from $N+1$ to $N_1$, let us replace $F_0$ inside this stripe by the map~$F$ from Lemma~\ref{l:incaps}. Denote the new map by~$F$ as well; then $F = F_0$ outside the inion of the aforementioned stripes by construction. Since independent stripes do not intersect, 
\begin{itemize}
\item $F$ is $C^1$-smooth;
\item $\dist_{C^1}(F_0, F) \le C\dist_{C^1}(F_0, \FPL) < C\delta$ (due to the choice of $N$  in the construction of $F_0$); 
\item $\dist_{C^0}(F, \FPL) < \gamma$.
\end{itemize}

If $\delta$ is sufficiently small, the second property implies that $F$ is a diffeomorphism (because $F_0$ is). Recall that $\dist_{C^1}(F_0, \FI) < \di$. Thus, for a small $\delta$ we have $\dist_{C^1}(F, \FI) < \di$. Due to the choice of~$\di$, the map $F$ satisfies properties~\ref{i:bilip} and~\ref{i:expand} from the definition of the class~$\CO$: property~\ref{i:bilip} follows from~\eqref{e:di}, whereas property~\ref{iv:expand} from Proposition~\ref{p:Finit} and the fact that $\di< 0.1$ imply that $F$ dilates with a factor at least 1.1 in restriction to any vertical segment, which, combined with property~\ref{i:bilip}, yields property~\ref{i:expand}. Property~\ref{i:skew} holds by Lemma~\ref{l:incaps}. Arguing as in the proof of Proposition~\ref{p:induction}, one can check that~$F$ satisfies the rest of the conditions from the definition of the class~$\CO$. Thus, $F \in \CI$.
The proof of Lemma~\ref{l:dense} is complete.

%% file: F_initExists.tex
\section{\texorpdfstring{$\FI$}{F-init} exists} \label{s:Finit}
In this section we will prove Proposition~\ref{p:Finit}. Let us recall it for convenience.
\newtheorem*{p:Finit}{Proposition~\ref{p:Finit}}
\begin{p:Finit}
For $\NI$ large enough, there exists a $C^1$-smooth Anosov diffeomorphism $\FI: \TT \to \TT$ such that:
\begin{enumerate}
\item
$\FI=\FB$ on $UH_0 \cup UH_1$. Thus, $\FI$ has a semi-thick horseshoe $\HS$.
\item
$\FI=\FL$ outside $\tilde R$. 
\item $\FI(UH_i)=UV_i$. 
\item $\FI(\UK) \cap \UK = UV_0 \cup UV_1$.
\item the unstable fibers of $\FI$ are strictly vertical.
\item $\FI$ expands the vertical fibers by a factor of at least $1.2$.
\item For any point $x \in \TT$ the image of the interior of the cone $C_H(x)$ under $d\FI$ covers the closure of the cone $C_H(\FI(x))$.
\end{enumerate}
\end{p:Finit}

\subsection{Bowen's map}
On any segment a Cantor set can be constructed using the following procedure. At the first step let us remove from our segment a concentric interval of length $a_1$. Then, at step two, remove from each of the remaining two segments a concentric interval of length $a_2/2$. Given a sequence $(a_j)_{j\in\NN}$, we continue this process to obtain a Cantor set in the limit: on step $i$ we remove $2^i$ intervals of length $2^{-i}a_i$. This set will be called \emph{a symmetric Cantor set}. We will consider only Cantor sets such that, first, $\sum\limits_{i=1}^\infty a_i$ is less than the length of the original segment (and therefore, the measure of our Cantor set will be positive), and second, the sequence $a_i$ decreases and $\lim\limits_{n\to \infty}a_n/a_{n+1}=1$. For example, one can take $a_n = \frac C {n^2}$. We will call symmetric Cantor sets with these two properties \emph{Bowen sets}. We will need the following elementary property of such sets:

\begin{prop} \label{p:equal_measures}
Let $\K$ be a Bowen set. Then the restriction $\Leb_{\K}$ of the Lebesgue measure on $\K$, rescaled to be a probability measure, is equal to the image $\nu$ of the Bernoulli measure on $\{0, 1\}^{\mathbb N}$ under the natural encoding map $\chi : \{0, 1\}^{\mathbb N} \to \K$. 
\end{prop}
\begin{proof}
Given any finite word of ones and zeroes $w$, set $I_w \subset \K$ to be a subset of $\K$ that consists of points $x$ such that $\chi^{-1}(x)$ begins with~$w$. Since $\K$ is symmetric, we have $\Leb_{K}(I_w) = 2^{-|w|}$. It is clear that $\nu(I_w)$ also equals $2^{-|w|}$. Since the sets $I_w$ generate the Borel $\sigma$-algebra on~$\K$, we have $\Leb_{\K} = \nu$.
\end{proof}

Consider a Bowen set on~$[a, b]$. Denote the first removed interval of the complement by~$(x_1, y_1)$. Our Cantor set naturally splits into two halves: one to the left from the interval $(x_1, y_1)$ and another one to the right. Let \emph {the Bowen map} from $[a,x_1]$ to $[a,b]$ be defined by the following conditions: it is $C^1$-smooth and it maps the left half of the Cantor set to the whole Cantor set and the intervals removed at step $n$ ($n>1$) from the left half to the intervals removed from the whole Bowen set at step $(n-1)$, preserving the order and orientation. Note that by defining the map at the endpoints of those intervals we define it on the whole left half of the Cantor set by continuity, because the endpoints are dense in the Cantor set. The Bowen map from $[y_1,b]$ to $[a,b]$ is defined analogously.

From the Bowen thick horseshoe construction in~\cite{Bow} one can easily deduce the following fact:
\begin{prop}
For any Bowen set there is a Bowen map, and its derivative at the points of the Bowen set equals~2. This map can be chosen so that at the points of the complement the derivative be gretater than~2. %\todogr{Откуда взялось, что производная больше 2?}
%\todo{$a_n$ убывает, число отреков уровня $n$ это $2^n$. Поэтому отрезки следующего уровня хотя бы вдвое короче отрезков предыдущего уровня. Но я бы не стал включать это в текст.}
% заданы два непересекающихся подотрезка $[a,x]$ и $[y,b]$ равной длины. Тогда существует канторовское множество положительной меры $C_{thick}\subset [a,x]$ и отображение $F_{Bow}\in C^1$ со следующими свойствами: центральный интервал канторовского множества переходит в $[x,y]$, множество $C_{thick}$ переходит в объединение $C_{thick}\cup C_{thick,2}$, где $C_{thick,2}$ - сдвиг исходного канторовского множества такой, что оно лежит на отрезке $[y,b]$. 
\end{prop}

\subsection{The map~\texorpdfstring{$\FB$}{F-Bow}} \label{s:FB}
In section~\ref{s:construction} we have briefly described the map~$\FB$ on the rectangles~$UH_0$ and~$UH_1$. In the present section we will describe this map in detail. In order to construct our $\FI$ we will also need to extend $\FB$ to the whole two-torus~$\TT$.

Let us define the horizontal stripe $RH_0$ (see fig.~\ref{f:rhi}), such that $UH_0 \subset RH_0 \subset \tilde R$, in the following way. First slightly enlarge $UH_0$ upwards and downwards to obtain a stripe $\tilde RH_0 \subset \tilde R$. The vertical gap between the horizontal boundaries of $\tilde RH_0$ and $UH_0$ should be less than~$\kappa$, and the number~$\kappa$ will be specified below in section~\ref{s:cones}. To obtain~$RH_0$, we continue $\tilde RH_0$ to the left and to the right until we reach the vertical boundary of~$\tilde R$. We define the stripes $\tilde RH_1 \supset UH_1$ and $RH_1 \supset \tilde RH_1$ analogously. 

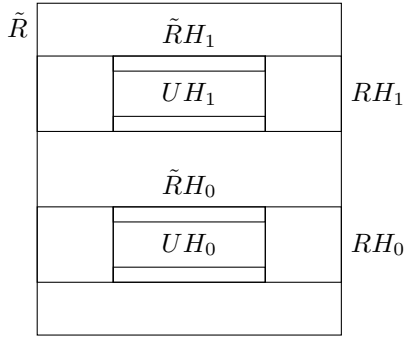
\begin{figure} 
\begin{center}
\begin{tikzpicture}[scale=1]

\draw (0,0) rectangle (2,1);
\draw (0,0) rectangle (-1,1);
\draw (2,0) rectangle (3,1);
\draw (0,0) rectangle (2,0.2);
\draw (0,1) rectangle (2,0.8);
\node at (1, 0.5) {$UH_0$};
\node [above] at (1, 1) {$\tilde RH_0$};
\node [right] at (3, 0.5) {$RH_0$};

\def \ys{-2}
\draw (-1,-0.7) rectangle (3,3.7);
\node [below left] at (-1, 3.7){$\tilde R$};

\draw (0,0-\ys) rectangle (2,1-\ys);
\draw (0,0-\ys) rectangle (-1,1-\ys);
\draw (2,0-\ys) rectangle (3,1-\ys);
\draw (0,0-\ys) rectangle (2,0.2-\ys);
\draw (0,1-\ys) rectangle (2,0.8-\ys);
\node at (1, 0.5-\ys) {$UH_1$};
\node [above] at (1, 1-\ys) {$\tilde RH_1$};
\node [right] at (3, 0.5-\ys) {$RH_1$};

%\draw [fill=gray!20, draw=black] (0,3) rectangle (8,4);
%\draw[pattern=north west lines, pattern color=gray] (1.5 ,0) rectangle (3.3 ,4);
%\draw (0, 1) -- (0, 3);
%\node [left] at (0, 0.5) {$UH_0$};
%\draw [fill] (6, 0) circle[radius=2pt];
\vspace{-15mm}
\end{tikzpicture}
\end{center}
\caption{$\tilde R$, $UH_i$, $\tilde RH_i$, $RH_i$.}
\label{f:rhi}
\end{figure}

The map $\FB$ is going to coincide with $\FL$ outside $RH_0 \cup RH_1$, and inside those two stripes $\FB$ is obtained from $\FL$ by a certain surgery. 
We define $\FB$ inside $RH_0$, and in $RH_1$ the construction is fully analogous. In restriction to $RH_0$ the map $\FL$ is a direct product of an affine horizontal contraction and an affine vertical dilation $f_L$. The map $\FB$, in restriction to $RH_0$, will be  a direct product of the same affine horizontal contraction and a nonlinear vertical dilation $f_{Bow}$ that will be described in a moment.

Denote by $Q_0$ and $Q_1$ the projections of $UH_0$ and $UH_1$ to the vertical axis, and let $Q$ be the convex hull of~$Q_0 \cup Q_1$. Let $I_1 = Q \setminus (Q_0\cup Q_1)$. Choose a Bowen set~$\CT$ of positive measure included into~$Q$ and such that the interval removed at the first step coincides with~$I_1$. Then, by the previous section, there is a Bowen map $f_{Bow} : Q_0 \to Q$, constructed as above. Let $RQ_0 \supset Q_0$ be the projection of $RH_0$ to the vertical axis. Continue the map $f_{Bow}$ to $RQ_0$ in such a way that in a vicinity of the boundary of $RQ_0$ it coincide with~$f_L$. This ensures that $\FB$ will be smooth on~$\dh RH_0$. Since $f_L' > 2$, we can assume that $f_{Bow}' \ge 2$ on $RQ_0$.

The construction implies that 
\begin{itemize}
\item
$\FB(UH_i) = \FL(UH_i)$ and $\FB(RH_i) = \FL(RH_i)$ for $i=0, 1$,
\item 
$\FB$ is a bijection,
\item 
$\FB$ has discontinuities on $\dv RH_0 \cup \dv RH_1$ and is smooth everywhere else.
\end{itemize}

\subsection{Construction of~\texorpdfstring{$\FI$}{F-init}}
Were it smooth, the map $\FB$ could play the role of~$\FI$. So, we are going to obtain $\FI$ by smoothing $\FB$ inside $(RH_0 \setminus \tilde RH_0) \cup (RH_1 \setminus \tilde RH_1)$ (that is, in the vicinity of the vertical boundaries of the stripes on which the surgery has been performed to construct~$\FB$). 

Consider the coordinates $Oxy$ on $\tilde R$ such that the $x$-axis is, as usual, horizontal and the $y$-axis is vertical. Choose a smooth bump function $\varphi(x)$ that equals zero on the horizontal projection of~$\UK$, equals one in the neighborhood of the boundary of the horizontal projection of~$\tilde R$, and satisfies $|\varphi'| < 200$. The last condition is satisfiable because $\pi_{hor}(\tilde R) \setminus \pi_{hor}(\UK)$ consists of two segments of length~$0.01$; here $\pi_{hor}$ is the horizontal projection.

Define the map $\FI$ inside $RH_0 \cup RH_1$ by the formula
\begin{equation}
\FI(x,y)=\varphi(x)\FL(x,y)+(1-\varphi(x))\FB(x,y), \label{e:phy}
\end{equation}
and set $\FI = \FL$ outside this union.
Inside $RH_0 \cup RH_1$ our map is smooth, because there both $\FB$ and $\FL$ are smooth and $\varphi$ is smooth everywhere on its domain. Near the boundary of the set $RH_0 \cup RH_1$ $\FI$ equals $\FL$ and therefore is smooth, too. The maps $\FB$ and $\FL$ both take pieces of the vertical fibers cut out by $\tilde R$ monotonically into the pieces of the vertical fibers cut out by $\FL(\tilde R)$. Therefore, $\FI$ is bijective in restriction to any such fiber (as a convex combination of the restrictions of $\FB$ and $\FL$), and hence, in restriction to~$\tilde R$. Since outside $\tilde R$ the map $\FI$ coincides with~$\FL$, it is bijective on the whole torus~$\TT$. %Поскольку $|d\FI| \ne 0$, \todogr{это очевидно?}получаем, что $\FI$ --- диффеоморфизм.

In section~\ref{s:cones} below we will check that $\FI$ is an Anosov diffeomorphism and prove that condition~\ref{iv:cones} from Proposition~\ref{p:Finit} holds for it. Condition~\ref{iv:only-V} is satisfied because, by construction, $\FL(\UK) = \FB(\UK) = \FI(\UK)$.
The rest of the properties declared in Proposition~\ref{p:Finit} follow straightforwardly from the construction of~$\FI$.

\subsection{\texorpdfstring{$F_{init}$}{F-init} is an Anosov diffeomorphism} \label{s:cones}

Set $\psi(x, y) = \phy(x)$ inside $\tilde R$ and $\psi(x,y) = 1$ outside~$\tilde R$. Then the Leibniz formula gives, at every point in $\tilde R$,
\[
d\FI = \psi d\FL + (1-\psi)d\FB \quad  + \quad (\FL-\FB)d\psi.
\]
Here $d\psi$ is a row vector and $\FL-\FB$ is a column vector, so their product is a $2 \times 2$ matrix.

Fix some point $p\in \tilde R$ and denote $$A=\psi(p) d_p\FL + (1-\psi(p))d_p\FB, \;\; B=(\FL(p)-\FB(p))d_p\psi.$$ Then for large $\NI$

\begin{itemize}
\item $ A = \bigl(\begin{smallmatrix}
a_1 & 0 \\ 0& a_2
\end{smallmatrix} \bigr),$ with $ a_1 \in (0, 0.5), \; a_2 \ge 2$;
\item $\|B\| < c$, for some small number $c$ independent of~$p$, and $B$ has the form $\bigl(\begin{smallmatrix}
0 & 0 \\ \delta & 0
\end{smallmatrix} \bigr)$.
\end{itemize}

Let us check the first claim. For $\FL$ and $\FB$ the vertical and horizontal directions are eigendirections, and there is contraction in the horizontal direction and an expansion in the vertical direction with a factor at least two. Therefore, the same is true for their convex combination~$A$. 

%Now note that the first component of $\FL-\FB$ is zero and so the first row of the matrix $B$ consists of zeros 

Since in $\tilde R$ the maps $\FL$ and $\FB$ permute the vertical fibers in the same way, $\FL-\FB = (\begin{smallmatrix} 0 \\ * \end{smallmatrix})$.   
Since $\psi(x, y) = \phy(x)$, we have $d\psi = (* \; 0)$, so $B = \bigl(\begin{smallmatrix}
0 & 0 \\ * & 0
\end{smallmatrix} \bigr)$.
Let us show that for a large $\NI$ we have $\|B\| < c$. By definition, $B=(\FL-\FB)d\psi$.
The factor $d\psi$ is bounded by the construction of $\psi$: $\psi$ does not depend on $y$, and $|\psi'_x| = |\phy'| < 200$, so it suffices to estimate the difference $\FL-\FB$. The maps $\FL$ and $\FB$ preserve the vertical foliation and permute the vertical fibers in the same way. Since they coincide on $\dh RH_i$, the factor $\FL-\FB$ is bounded by the vertical size (that is, the diameter of the projection to the vertical axis) of $\FL(RH_i)$. Since $\FL(UH_i) = UV_i$, the vertical size of $\FL(UH_i)$ equals the vertical size of~$\KK$. If the number~$\kappa$ introduced in section~\ref{s:FB} (recall that it is the size of the vertical gap between the boundaries of $UH_i$ и $RH_i$) is small enough, the vertical size of $\FL(RH_i)$ is smaller than the doubled vertical size of~$\KK$. Since for a large $\NI$ the rectangle $\KK$ is very small, we have $\|B\| < c$.

Thus, we have $d\FI(p) = A + B = (\begin{smallmatrix}
a_1 & 0 \\ \delta & a_2
\end{smallmatrix} )$. This lower triangular matrix is obviously non-degenerate. One of its eigendirections is purely vertical, and another one is almost horizontal: indeed, the eigenvector has the form $(a_2-a_1, \, -\delta)$, and $\delta$ is small, while $|a_2-a_1|\ge 1.5$.

Furthermore, it is straightforward to check that $d_p\FI$ takes the complement of the cone $C_H(p)$ inside itself, so condition~\ref{iv:cones} from Proposition~\ref{p:Finit} holds at~$p$. Indeed, $d\FI(x)(u,v) = (a_1u, \delta u + a_2 v)$, and if $|u|<|v|$, then we have $$|a_1u|<|a_1v|<\gamma(|a_2 v| - |\delta v|) < \gamma|a_2 v + \delta u|,$$
for some $\gamma\in(0,1)$.  
Likewise, it is straightforward to show that for the vertical and the horizontal cone fields $C_{\alpha, V} = \{(u,v)\mid |u| \le \alpha |v|\}$ and $C_{\alpha, H} = \{(u,v)\mid |v| \le \alpha |u|\}$ of some small aperture $\alpha$ (however, $\delta$ should always be small relative to this $\alpha$) the cone condition holds, i.e., the cones are expanded and mapped inside the cones of the same field under $d\FI$ or $d\FI^{-1}$, respectively. Invertibility of $d\FI$ outside $\tilde R$ is beyond doubt, as well as the fact that the cones condition holds there for the same cone fields, because there we have $\FI = \FL$. Thus we have established that $\FI$ is an Anosov diffeomorphism.

%% file: calculus.tex
\section{Proofs of the two technical lemmas} \label{s:technical}
In this section we will prove the two lemmas stated above.
\newtheorem*{l:delta}{Lemma~\ref{l:delta}}
\begin{l:delta}
Consider an arbitrary homeomorphism $F_0 \in \CI$ and an arbitrary roughly horizontal stripe~$\Pi$. Suppose that for some $\delta > 0$ for any points $p, q \in \Pi$ that lie on the same vertical segment we have
$\|dF_0(p) - dF_0(q)\| < \delta.$
Then $\dist_{\Pi}(F_0, \Lcal_\Pi(F_0)) < \sqrt{5}\delta$.
\end{l:delta}

\begin{proof}
Denote $\Lcal_\Pi(F_0)$ by $\FPLL$ and the horizontal and the vertical components of $F_0$ by $g_0$ and $f_0$ respectively:
\[
F_0: (x, y) \mapsto (g_0(x, y), f_0(x, y)).
\]
Then for any $p, q \in \Pi$ that lie on the same vertical interval we have $\|df_0(p) - df_0(q)\| < \delta$. Let $\fpl$ be the vertical component of $\FPLL$. Since the horizontal components of $F_0$ and $\FPLL$ are equal, we can write
\[
\|dF_0(p) - d\FPLL(p)\| = \|df_0(p) - d\fpl(p)\|.
\]

Let $v$ and $h$ be the vertical and horizontal unit (constant) vector fields on the neighborhood of $\Pi$. Consider any vertical section $I$ of the stripe~$\Pi$. The derivative $\nabla_v(\fpl)|_I$ is constant and equals the average value of $\nabla_v(f_0)$ on $I$. Hence, for any $p \in I$ we have
\begin{equation} \label{e:vertical}
|\nabla_v(f_0)(p) - \nabla_v(\fpl)(p)| \le \max_{q \in I}\|df_0(p) - df_0(q)\| < \delta.
\end{equation}
Let $a$ be the upper endpoint of $I$ and $w \in T_a \TT$ be the unit vector tangent to the boundary of $\Pi$. Since the restrictions of $f_0$ and $\fpl$ to $\partial_h\Pi$ coincide, we see (extrapolating $\fpl$ outside $\Pi$ linearly) that $\nabla_w(\fpl)(a)=\nabla_w(f_0)(a)$. Since the stripe is roughly horizontal, we may write $w = (c_v v + c_h h)(a)$, where $|c_h| > 1/\sqrt{2} > |c_v|$. Thus we have
\[c_h\nabla_h(\fpl)(a) + c_v\nabla_v(\fpl)(a) = \nabla_w(\fpl)(a) = \nabla_w(f_0)(a) = c_h\nabla_h(f_0)(a) + c_v\nabla_v(f_0)(a).
\]
Using inequality~\eqref{e:vertical}, we obtain the following estimate:
\[
|\nabla_h(\fpl)(a) - \nabla_h(f_0)(a)|
= \frac{|c_v|}{|c_h|}|\nabla_v(\fpl)(a) - \nabla_v(f_0)(a)|
< \frac{|c_v|}{|c_h|}\delta < \delta.
\]
A similar inequality holds at the lower endpoint $b$ of $I$. Take a number $s$ such that $\nabla_h(f_0)(z) \in (s-\delta/2, s+\delta/2)$ for any $z \in I$. Then both $\nabla_h(\fpl)(a)$ and $\nabla_h(\fpl)(b)$ are in $(s-3\delta/2, s+3\delta/2)$. 
Since the map $\fpl$ is linear on vertical intervals, we may represent it as
\[
\fpl(x, y) = \alpha(x)y + \beta(x),
\]
\[
\nabla_h(\fpl)(x, y) = \alpha'(x)y + \beta'(x).
\]
Thus the function $\nabla_h(\fpl)(x, y)$ is linear in $y$ on $I$, and therefore, $\nabla_h(\fpl)(I) \subset (s-3\delta/2, s+3\delta/2)$. Thus,
$$\nabla_h(f_0)(I) \subset (s-\delta/2, s+\delta/2),\quad\nabla_h(\fpl)(I) \subset (s-3\delta/2, s+3\delta/2),$$
which implies
$|\nabla_h(f_0)(p) - \nabla_h(\fpl)(p)| < 2\delta$. Using~\eqref{e:vertical}, we obtain the required estimate: 
\[
\|dF_0(p) - d\FPLL(p)\| = \|df_0(p) - d\fpl(p)\| < \sqrt{5}\delta.
\]
\end{proof}

\newtheorem*{l:incaps}{Lemma~\ref{l:incaps}}
\begin{l:incaps}
There exists a universal constant $C > 0$ such that the following holds. Consider an arbitrary $F_0\in\CI$ and an arbitrary roughly horizontal stripe~$\Pi$. For any $\gamma > 0$ there exists $F\in C^1(\TT, \TT)$ such that
\begin{itemize}
\item $F = F_0$ outside $\Pi$;
\item $\dist_{C^1}(F_0|_{\Pi}, F|_{\Pi}) \le C\cdot\dist_{C^1}(F_0|_{\Pi}, \FBL)$;
\item $\dist_{C^0}(\FBL, F|_{\Pi}) < \gamma$;
\item $F$ preserves the vertical foliation.
\end{itemize} 
\end{l:incaps}

\begin{proof}
Consider the map $\FBL$. First, $F_0|_{\dh\Pi} = \FBL|_{\dh\Pi}$, because $\FBL$ is obtained from~$F_0$ by linearization on~$\Pi$. Second, $F_0|_{\dv\Pi} = \FBL|_{\dv\Pi}$, because on $\dv\Pi$ the map~$F_0$ is already linear.
Since our stripe is roughly horizontal, in the $(x, y)$ coordinates on~$R$ its boundary consists of the graphs of two smooth functions $\varphi_1,\varphi_2\colon\varphi_1(x) < \varphi_2(x)$. Consider a function $\rho\in C^\infty(\mathbb R)$ such that $\rho(t) = 0$ when $t<0$, $\rho(t) = 1$ when $t > 1$, and $\rho$ is monotonically increasing on $[0,1]$. Fix some small $\alpha > 0$ and define $F$ on $\Pi$ as follows:
\begin{multline*}
F(x,y) = \rho\left(\frac{\varphi_2(x) - y}{\alpha}\right)\rho\left(\frac{y -\varphi_1(x)}{\alpha}\right)\FBL +
 \left(1 - \rho\left(\frac{\varphi_2(x) - y}{\alpha}\right)\rho\left(\frac{y -\varphi_1(x)}{\alpha}\right)\right)F_0.
\end{multline*}
Obviously, $F$ preserves the vertical foliation.
The map $F|_{\overline\Pi}$ is $C^1$-smaooth and coincides with $F_0$ on $\partial\Pi$. If the vertical distance from the point $(x, y)\in \Pi$ to $\dh\Pi$ is greater than $\alpha$, then $F(x, y) = \FBL(x, y)$.  

Let us show that $$\dist_{C^1}(F_0|_{\Pi}, F|_{\Pi}) \le C\cdot\dist_{C^1}(F_0|_{\Pi}, \FBL).$$ Consider $\delta$ such that $\|dF_0 - d\FBL\| < \delta$ on $\Pi$; here $\delta$ is not necessarily small. First, we will estimate the difference of $y$-derivatives for $F_0$ and $F$. For short, let us use the notation $\rho_1(x,y) = \rho\left(\frac{y -\varphi_1(x)}{\alpha}\right), \; \rho_2(x,y) = \rho\left(\frac{\varphi_2(x) - y}{\alpha}\right)$. Outside $\Pi$ set $F = F_0$. 
Outside the neighborhood of the set $\dh\Pi$ we have $\rho_1 = \rho_2 = 1$ and
\[  \|F'_y - (F_0)'_y\| = \|(\FBL)'_y - (F_0)'_y\| < \delta.\]
Inside the neighborhood of the upper boundary of $\Pi$ we have $\rho_1 \equiv 1$, and therefore, $F = \rho_2\FBL + (1-\rho_2)F_0$, which yields
$$\|F'_y - (F_0)'_y\| = \|(\rho_2)'_y(\FBL - F_0) + \rho_2 \cdot ((\FBL)'_y - (F_0)'_y)\|\le$$
$$\le \frac{1}{\alpha}\cdot\max|\rho'|\cdot\|\FBL - F_0\| + \|(\FBL)'_y - (F_0)'_y\|.$$
Yet again, $\|(\FBL)'_y - (F_0)'_y\|\le \delta.$ This inequality also implies that if the vertical distance from the point $(x,y)$ to the upper boundary of the stripe is less than $\alpha$, we have the estimate $\|\FBL - F_0\| \le \alpha\delta$. Thus, in the vicinity of the upper boundary we have $\|F'_y - (F_0)'_y\| \le C_y\delta$, where $C_y = \max|\rho'| + 1$ is a constant independent of $F_0$ and the choice of~$\Pi$. Likewise we can obtain an analogous estimate in the vicinity of the lower boundary of~$\Pi$. Furthermore, an analogous argument is applicable to the difference of the derivatives in $x$ in the vicinity of $\partial_h\Pi$, the only difference is that the derivatives of $\varphi_j$ come into play:
$$\|F'_x - (F_0)'_x\| \le (1 + \max|\rho'|\cdot \max_x(|\varphi'_1(x)|, |\varphi'_2(x)|))\cdot\delta = C_x\cdot\delta.$$
Since the stripe $\Pi$ is roughly horizontal, $|\varphi'_i(x)| \le 1$, so we can take $C_x = C_y$.
Thus, inside the stripe $\Pi$ we have an estimate $\|dF - dF_0\| \le C_x\delta$, where the constant $C_x$ is fixed and does not depend on the stripe $\Pi$ or on the map $F_0$.
Since the width of the stripe is bounded by the diameter of the torus and on the boundary of the stripe we have $F = F_0$, the distance between $F$ and $F_0$ in $C^0$ is also at most $C_x \cdot \delta$. 
Therefore, we get an estimate $\dist_{C^1}(F_0|_{\Pi}, F|_{\Pi}) \le C\cdot\dist_{C^1}(F_0|_{\Pi}, \FBL)$ with $C = 2 C_x$.

Note now that by choosing $\alpha$ small enough we can make sure that the restrictions of $F$ and $\FBL$ to $\Pi$ are $\gamma$-close in $C^0$. 
\end{proof}